\documentclass[reqno,12pt]{amsart}

\usepackage{amssymb,graphicx,color}

\usepackage[colorlinks=true]{hyperref}
\hypersetup{urlcolor=blue, citecolor=red}

\theoremstyle{plain}
\newtheorem{thm}{Theorem}[section]
\newtheorem{lem}{Lemma}[section]
\newtheorem{prop}{Proposition}[section]
\newtheorem{cor}{Corollary}[section]
\newtheorem{defs}{Definition}[section]
\theoremstyle{definition}
\newtheorem{rmk}{Remark}[section]

\newcommand{\NN}{{\mathbb{N}}}

\newcommand{\RR}{{\mathbb{R}}}

\newcommand{\QQ}{{\mathbb{Q}}}
\newcommand{\bu}{\mathbf{u}}
\newcommand{\bv}{\mathbf{v}}
\newcommand{\bw}{\mathbf{w}}

\newcommand{\bx}{\mathbf{x}}

\newcommand{\bbf}{\mathbf{f}}

\newcommand{\bF}{\mathbf{F}}

\newcommand{\bnabla}{\boldsymbol{\nabla}}

\newcommand{\define}{\stackrel{\text{\rm def}}{=}}
\newcommand{\Ccal}{{\mathcal C}}

\newcommand{\Bcal}{{\mathcal B}}
\newcommand{\Fcal}{{\mathcal F}}

\newcommand{\Ncal}{{\mathcal N}}

\newcommand{\Gcal}{{\mathcal G}}
\newcommand{\Kcal}{{\mathcal K}}

\newcommand{\Vcal}{{\mathcal V}}

\newcommand{\loc}{{\text{\rm loc}}}

\newcommand{\Mcal}{{\mathcal M}}
\newcommand{\Dcal}{{\mathcal D}}
\newcommand{\Tcal}{{\mathcal T}}
\newcommand{\Ucal}{{\mathcal U}}

\newcommand{\Lcal}{{\mathcal L}}
\newcommand{\Ocal}{{\mathcal O}}

\newcommand{\rd}{{\text{\rm d}}}
\newcommand{\rw}{{\text{\rm w}}}

\newcommand{\rb}{{\text{\rm b}}}

\newcommand{\rc}{{\text{\rm c}}}

\newcommand{\ddt}[1]{\frac{\text{\rm d}#1}{\text{\rm dt}}}

\newcommand{\co}{\operatorname{co}}

\newcommand{\Vinner}[1]{(\!({#1})\!)_{H^1}}
\newcommand{\dual}[1]{\langle{#1}\rangle}

\begin{document}
\numberwithin{equation}{section}


\title[Time-dependent statistical solutions of the 3D NSE]{
  Properties of time-dependent statistical solutions of the
  three-dimensional Navier-Stokes equations}

\author[C. Foias]{Ciprian Foias${}^1$}
\author[R. Rosa]{Ricardo M. S. Rosa${}^2$}
\author[R. Temam]{Roger Temam${}^3$}

\address{${}^1$ Department of Mathematics, Texas A\&M University,
  College Station, TX 77843, USA.}
\address{${}^2$ Instituto de Matem\'atica, Universidade Federal do Rio de Janeiro, 
  Caixa Postal 68530 Ilha do Fund\~ao, Rio de Janeiro, RJ 21945-970,
  Brazil.}
\address{${}^3$ Department of Mathematics, Indiana University, Bloomington,
  IN 47405, USA}
  
\email[R. Rosa]{rrosa@im.ufrj.br}
\email[R. Temam]{temam@indiana.edu}

\date{November 22, 2012}

\thanks{This work was partly supported by the National Science Foundation under the 
grants NSF-DMS-0604235, NSF-DMS-0906440, and NSF-DMS-1109784, by the Research Fund of Indiana University, and by CNPq, Bras\'{\i}lia, Brazil, under the grants 490124/2009-7, 307077/2009-8 and 500437/2010-6.}

\thanks{Accepted for publication in \emph{Annales de l'Institute Fourier.}}


\maketitle

\begin{center}
  \emph{This article is dedicated to the memory of Mark Vishik}
\end{center}

\begin{abstract}
This work is devoted to the concept of statistical solution of the Navier-Stokes equations, proposed as a rigorous mathematical object to address the fundamental concept of ensemble average used in the study of the conventional theory of fully developed turbulence. Two types of statistical solutions have been proposed in the 1970's, one by Foias and Prodi and the other one by Vishik and Fursikov. In this article, a new, intermediate type of statistical solution is introduced and studied. This solution is a particular type of a statistical solution in the sense of Foias and Prodi which is constructed in a way akin to the definition given by Vishik and Fursikov, in such a way that it possesses a number of useful analytical properties.
\end{abstract}

\bigskip\bigskip

\section{Introduction}

Our aim in this work is to address the mathematical formulation of the concept of statistical 
solution of the three-dimensional Navier-Stokes equations for incompressible fluids. 
Statistical solutions have been introduced as a rigorous mathematical object to formalize the 
notion of ensemble average in the conventional statistical theory of turbulence. 

In turbulent flows, physical quantities vary rapidly and erratically in space and time but are 
somewhat well-behaved in a statistical sense, when averaged in some form. Averages 
might be taken over a certain time interval, over a certain region in space, and over an 
ensemble of flows (e.g. a number of experiments in a wind tunnel under seemingly the same 
conditions). It is this latter average which is called ensemble average. 

The conventional theory of turbulence has relied in most part on empirical evidence and 
heuristic arguments \cite{reynolds1895, taylor1935, taylor1938, kolmogorov1941a, 
kolmogorov1941b, batchelor1953, hinze1975, moninyaglom1975, frisch1995, lesieur1997}. 
More recently, a number of rigorous results have been 
obtained for mean quantities of three-dimensional flows, usually either from time 
averages of weak solutions (e.g. \cite{howard72, bcfm1995, constantindoering94, constantindoering95, doeringtiti1995, fjmrt05}) or from statistical solutions (e.g. 
\cite{hopf1952, foias74, fmrt2001b, fmrt2001c,basson2006, rrt08}). It is therefore our 
belief that a better understanding of statistical solutions are of fundamental importance for a 
rigorous mathematical approach to the theory of turbulence.

There are two main notions of statistical solutions, one introduced by Foias and Prodi 
\cite{foias72,foiasprodi76} and the other by Vishik and Fursikov \cite{vishikfursikov78,vishikfursikov88}
(see also an earlier related mathematical work by Hopf \cite{hopf1952}). 
In the present work we essentially formulate a modified definition of Vishik-Fursikov statistical solutions in slightly different form than their original definition, and which becomes a particular case of a statistical solution in the sense of Foias and Prodi which is more amenable to analysis and has a number of useful properties. 

A statistical solution as defined in \cite{foias72,foiasprodi76} is a family of Borel measures
parametrized by the time variable and defined on the phase space of the Navier-Stokes 
equations, representing the probability distribution of the velocity field of the flow at
each time (Definition \ref{deftimedependetstatisticalsolution}). 
The definition given in \cite{vishikfursikov78,vishikfursikov88}, in its turn, is that of
a single Borel measure on the space of trajectories, representing the probability distribution
of the space-time velocity field (Remark \ref{rmkorgdefvfss}).

The phase space considered here for the Navier-Stokes equations, which 
we denote by $H$, is the space of square-integrable divergence-free velocity fields with 
the appropriate boundary conditions. We consider either no-slip boundary 
conditions on a bounded set $\Omega\subset\RR^3$ with smooth boundary, 
or periodic boundary conditions on a domain $\Omega=\Pi_{i=1}^3(0,L_i)$, $L_i>0$, 
$i=1,2,3$. It is assumed that the flow is forced by a given external field of 
possibly time-dependent volume forces with values in $H$ and essentially bounded in 
$H$ with respect to the time variable.
In the periodic case, it is also assumed that the space averages of the velocity and force 
fields are zero. 

The concept of Leray-Hopf weak solution is essential to our analysis and refers to a
weak solution in the classical sense of the Navier-Stokes equations which also 
satisfies a certain energy inequality and which is strongly continuous at the initial 
time. This energy inequality and the strong continuity at the initial time play a crucial role 
in our formulations of the statistical solutions. Leray-Hopf weak solutions are also
weakly continuous at any given time, a fact that led us quite naturally to make
extensive use of the weak topology of the space $H$ and of the associated 
Borel measures, allowing for the use of a number of results in measure theory which
are, however, quite delicate. This idea of using the weak continuity of weak solutions
in the study of statistical solutions is originally due to Prodi (see \cite{foiasprodi76}) and 
is exploited here extensively to our notion of Vishik-Fursikov statistical solution.

We first define a Vishik-Fursikov measure, which is more akin to their original definition 
of statistical solution, being indeed a measure in trajectory space. More precisely, 
in our case, a Vishik-Fursikov measure is a Borel probability measure in the 
space of weakly continuous functions and which is carried by the set of Leray-Hopf weak 
solutions (and with finite mean kinetic energy at each time). Then, our definition of 
Vishik-Fursikov statistical solution is that of the family of projections, in time, 
of a Vishik-Fursikov measure.

This definition of Vishik-Fursikov measure is to be compared with the original 
definition of statistical solution in the sense of Vishik and Fursikov 
\cite{vishikfursikov78,vishikfursikov88}. Their definition is that of a measure
on the space of trajectories satisfying a less sharp mean energy inequality 
and such that there exists a measurable subset of the set of weak solutions not
necessarily of Leray-Hopf type and which carries the measure.

One noticeable difference is that, in their definition, the measurability of the carrier is part of 
the assumption, while in our approach we prove that the set of Leray-Hopf weak solutions 
is measurable. 

The study of the topological structure of the set of Leray-Hopf weak solutions is indeed a 
delicate part of the work, which is related to the fact that it is not known whether the weak
solutions are unique or not. Among other similar results, we prove, in fact, that the set of 
Leray-Hopf weak solutions on an interval of the form $[0,T]$ is a 
$\Gcal_{\delta\sigma}$-subset of the set of weakly continuous functions on $[0,T]$ which 
are Leray-Hopf weak solutions on the interval $(0,T]$ (hence not necessarily strongly 
continuous at the initial time), and that this latter set is a $\sigma$-compact
subset of the space of weakly continuous functions on $H$, with the topology of 
uniform weak convergence (Corollary \ref{corgdeltasigma} and 
Proposition \ref{ucalisharpsigmacompact}).

The bounded subsets of the set of Leray-Hopf weak solutions on an interval of the form 
$[0,T]$ is a Borel set, but it is not necessarily compact, so we actually define a Vishik-Fursikov measure as a measure in the sequential closure of the set of Leray-Hopf weak 
solutions which have finite mean kinetic energy at each time and which are continuous 
at the initial time in a mean sense (Definition \ref{defvfmeasure}). The sequential closure 
of the set of Leray-Hopf weak solutions is the $\sigma$-compact 
set mentioned above, for which the bounded sets are compact. Compactness is usually a 
crucial step in the proof of existence of solutions of a number of different types of problems 
and here it manifests itself in the application of the Krein-Milman theorem, in which the proof 
of existence of Vishik-Fursikov measures is based (Theorem \ref{thmexistencevfmeasure}).

Then, another difference in our definition of Vishik-Fursikov measure is that the mean 
energy inequality is not an assumption but actually follows from the fact that the 
measure is carried by the set of Leray-Hopf weak solutions, which themselves satisfy 
an energy inequality. We prove, in fact, a strengthened form of mean energy inequality
(Theorem \ref{thmvfmsatisfiesenergyine}).

Then, as we mentioned above, projecting a Vishik-Fursikov measure, as defined here, to the
phase space, at each time, yields a family of measures which is a statistical solution in
the sense of Foias and Prodi (Theorems \ref{defvfsshalfclosedinterval},
\ref{defvfssopeninterval}, and \ref{thmexistencevfss}). This particular type of statistical
solution is termed a Vishik-Fursikov statistical solution (Definition \ref{defvfss}), and it is 
much more amenable to analysis.

Then, we prove a regularity result saying that a Vishik-Fursikov measure is actually
carried by the set of Leray-Hopf weak solutions (Theorem \ref{vfsscarriedbyus}). 

One natural question then is whether all statistical solutions are Vishik-Fursikov statistical
solutions. This would be true if the weak solutions of the three-dimensional Navier-Stokes 
equations were known to be unique, but since this is not known, the answer to this question
is not trivial. We give, however, in the case of families of statistical solutions 
with support uniformly bounded in $H$, an intrinsic characterization of Vishik-Fursikov 
statistical solutions as limits, at each given time, of convex combinations of Dirac measures 
carried by finite collections of Leray-Hopf weak solutions (Theorem \ref{charactvfr}).

In this work we only consider time-dependent statistical solutions. The mathematical
framework in the case of stationary statistical solutions and, in particular, of Vishik-Fursikov 
statistical solutions will be further developed in \cite{frtssp2}. In that work, Vishik-Fursikov 
stationary statistical solutions play a major role and are shown to have a
number of good properties. 

We believe the solid framework presented here will
be useful not only for the study of stationary statistical flows, such as  
turbulent flows in statistical equilibrium in time, but also in the study of time-dependent 
flows, such as in the case of decaying turbulence. 
The role of Vishik-Fursikov statistical solutions in the study of turbulent flows 
will be given elsewhere. The dependence of Vishik-Fursikov measures and
of Vishik-Fursikov statistical solutions on parameters will also be presented in
future works.

Although the theory presented here has been developed specifically 
in the context of the Navier-Stokes equations, it can serve as a model to 
treat other equations which have similar properties and the potential pathologies 
of the Navier-Stokes equations (e.g. the Boussinesq equations of thermo-hydraulics), such as mainly the possibility of a lack of uniqueness 
and the weak continuity in the phase space.

We intended to dedicate this article to Mark Vishik on his 90th birthday,
but we very sadly learned of his passing away shortly after the meeting in
his honor for his 90th birthday. We thus dedicate this article to
his memory, in deep appreciation of the person he was and in recognition of his
numerous outstanding contributions at the forefront of mathematics, and in
particular of the major role he played in developing a rigorous mathematical
framework for the theory of turbulence.

\section{Preliminaries}

\subsection{The Navier-Stokes equations and the mathematical setting}
\label{NSEsettingsec}

In this section we recall some classical results about individual solutions
of the three-dimensional Navier-Stokes equations, for which 
the reader is referred to 
\cite{leray,lady63,temam,temam3,constantinfoias,fmrt2001a}.

We consider the three-dimensional incompressible Navier-Stokes equations, 
which can be written as
\[ \frac{\partial \bu}{\partial t} - \nu \nabla \bu
       + (\bu\cdot\bnabla)\bu + \bnabla p = \bbf, \qquad
      \bnabla \cdot \bu = 0.
\]
The variable $\bu=(u_1,u_2,u_3)$ denotes the velocity vector field;
the term $\bbf$ represents the mass density of volume forces applied to
the fluid and is assumed given; the parameter $\nu>0$ 
is the kinematic viscosity; and
$p$ is the kinematic pressure. We denote the space variable by
$\bx=(x_1,x_2,x_3)$ and the time variable by $t$.

We allow two types of boundary conditions: periodic and no-slip. 
In the periodic case we assume the flow is periodic with period $L_i$
in each spatial direction $x_i$, $i=1,2,3$, and we set
$\Omega =\Pi_{i=1}^3(0,L_i)$. In this case, since the equations for averages are easy to solve (see e.g. \cite{temam3}), we also assume that the averages of the flow
and of the forcing term over $\Omega$ vanish, i.e.
\[ \int_\Omega \bu(\bx,t) \;\rd\bx = 0,
        \qquad \int_\Omega \bbf(\bx,t) \;\rd\bx = 0.
\]
In the no-slip case, we consider the flow on a bounded domain
$\Omega\subset\RR^3$ with smooth boundary $\partial \Omega$ (at least of class 
$\Ccal^2$), and it is assumed that $\bu=0$ on $\partial \Omega$. 
Other boundary conditions such as those for channel flows can be treated similarly.

In either the periodic or the no-slip case one obtains a functional equation
formulation for the time-dependent velocity field $\bu=\bu(t)$ corresponding, at each time $t$, to the function $\bx\in\Omega\mapsto \bu(\bx,t)$.
For the functions spaces, one starts, in the periodic case, 
with the space of test functions
\[ \Vcal=\left\{\bu=\bw|_\Omega; \;
     \parbox{4.2in}{ $\bw \in \Ccal^\infty(\RR^3)^3, \;\bnabla\cdot\bw=0,
                \;\int_\Omega \bw(\bx)\;\rd\bx = 0$,
        $\bw$ is periodic with period $L_i$ in each
       direction $x_i$.}
    \right\},
\]
while, in the no-slip case, one considers the test functions
\[ \Vcal=\left\{\bu\in \Ccal_\rc^\infty(\Omega)^3; \;
      \bnabla\cdot\bu = 0 \right\},
\]
where $\Ccal_\rc^\infty(\Omega)$ denotes the space of
infinitely-differentiable real-valued functions with compact support
in $\Omega$.

In each case the space $H$ is defined as the completion of $\Vcal$ under
the $L^2(\Omega)^3$ norm. The space $V$ is the completion of $\Vcal$ under
the $H^1(\Omega)^3$ norm. We identify $H$ with its dual and consider
the dual space $V'$, so that $V\subseteq H \subseteq V'$, with
the injections being continuous, and each space dense in the following
one. In the two cases we consider, we have in fact that $V$ is compactly included in $H$.

We denote the inner products in $H$ and $V$ respectively by
\[ (\bu,\bv)_{L^2} = \int_\Omega \bu(\bx)\cdot\bv(\bx) \;\rd\bx,
     \quad \Vinner{\bu,\bv} = \int_\Omega \sum_{i=1,2,3}
        \frac{\partial \bu}{\partial x_i}
          \cdot \frac{\partial \bv}{\partial x_i}\; \rd\bx,
\]
and the associated norms by $|\bu|_{L^2}=(\bu,\bu)_{L^2}^{1/2}$,
$\|\bu\|_{H^1}=\Vinner{\bu,\bu}^{1/2}$.

The duality product between $V$ and $V'$ is denoted by
\[ \dual{\bu,\bv}_{V',V},
\]
and coincides with the $L^2$ inner product if $\bu\in H$ and $\bv\in V$. The norm in $V'$ is
given by
\[ \|\bu\|_{V'}=\sup_{\substack{\bv\in V \\ \bv\neq 0}} 
    \frac{\dual{\bu,\bv}_{V',V}}{\|\bv\|_{H^1}}.
\]

We first define the Stokes operator as an operator $A:V\rightarrow V'$, by duality,
through, the formula
\[ \dual{A\bu,\bv}_{V',V} = \Vinner{\bu,\bv}, \quad \forall\bu,\bv\in V.
\]
The restriction of the operator $A$ to $D(A)=\{\bu\in V; A\bu\in H\}$ yields an unbounded
self-adjoint closed operator $A|_{D(A)}:D(A)\subset H\rightarrow H$ which
is positive definite and with compact inverse. Hence, it has a countable number of 
eigenvalues $\{\lambda_j\}_{j\in \NN}$, counted according to their multiplicity, in increasing
order, with each eigenvalue $\lambda_j$ associated to an eigenfunction $\bw_j$.
The \emph{Galerkin projector} onto the space spanned by the eigenfunctions
associated with the first $m$ eigenvalues is denoted by $P_m$.
Since we are assuming the domain to have a smooth boundary, we have the 
characterizations $D(A)=H^2(\Omega)\cap V$ and $A|_{D(A)}=-P_{LH}\Delta$, where 
$\Delta$ is the Laplacian and $P_{LH}$ is the Leray-Helmholtz orthogonal projector 
from $L^2(\Omega)^3$ onto $H$.

The following Poincar\'e inequality holds:
\begin{equation}
  \label{poincareineq}
  \lambda_1 |\bu|_{L^2}^2 \leq \|\bu\|_{H^1}^2,  \quad \forall\bu\in V,
\end{equation}
where $\lambda_1>0$ is the first eigenvalue of the Stokes operator.

We also consider the trilinear form
\[ b(\bu,\bv,\bw) = \int_\Omega (\bu\cdot\bnabla)\bv\cdot \bw \;\rd\bx, 
     \quad \bu,\bv,\bw\in V,
\]
continuously defined on $V$, which also defines, by duality, a bilinear operator 
$B:V\times V\rightarrow V'$ according to 
\[ \dual{B(\bu,\bv),\bw}_{V',V} = b(\bu,\bv,\bw), \quad\forall\bu,\bv,\bw\in V.
\]

Then, the functional equation takes the form
\begin{equation}
  \label{nseeq}
  \ddt{\bu} + \nu A\bu + B(\bu,\bu) = \bbf.
\end{equation}

Throughout this work we consider weak solutions and statistical solutions on a time
interval $I\subset \RR$ and for the sake of simplicity we make the following standing hypothesis on the forcing term:
\begin{equation}
  \bbf \in L^\infty(I,H).
\end{equation}

Given a subset $X$ of $H$, we also denote by $X_\rw$ this subset
endowed with the weak topology of $H$. In particular, $H_\rw$ 
denotes the space $H$ endowed with its weak topology.
The closed ball of radius $R$ in $H$ is denoted by $B_H(R)$.
Since $H$ is a separable Hilbert space, its weak topology is metrizable
on bounded sets, and in particular $B_H(R)_\rw$ is a completely metrizable 
metric space. 

With this framework set, we have the following definition of a Leray-Hopf weak
solution.

\renewcommand{\theenumi}{\roman{enumi}}
\begin{defs}
  \label{deflerayhopfweaksolution}
  A (Leray-Hopf) weak solution on a time interval $I\subset\RR$
  is defined as a function $\bu=\bu(t)$ on $I$ with values in $H$
  and satisfying the following properties:
  \begin{enumerate}
    \item \label{lhuinhv} $\bu\in L_\loc^\infty(I;H)\bigcap L_\loc^2(I;V)$;
    \item \label{lhutinvprime} $\partial\bu/\partial t \in L_\loc^{4/3}(I;V')$;
    \item \label{lhuinhw} $\bu\in \Ccal_\loc(I;H_w)$, i.e. $\bu$ is weakly continuous 
      in $H$, which means
      $t\mapsto (\bu(t),\bv)$ is continuous from $I$ into $\RR$, for
      every $\bv\in H$;
    \item \label{lhfnse} $\bu$ satisfies the functional equation \eqref{nseeq}
      in the distribution sense on $I$, with values in $V'$
      \footnote{This is equivalent to \[ (\bu(t),\bv)_{L^2} = (\bu(s),\bv)_{L^2}
        + \int_s^t\{(\bbf(\tau),\bv)_{L^2}- \nu \Vinner{\bu(\tau),\bv} 
         - b(\bu(\tau),\bu(\tau),\bv)\}\;\rd\tau, \] for every $t,s$ in $I$ and
      every $\bv$ in $V;$ see e.g. \cite[Ch. 3, Section 1]{temam}};
    \item \label{lheineq} 
      For almost all $t'$ in $I$, $\bu$ satisfies the following energy inequality:
      \begin{equation}
        \label{energyinequalityintegralform}
          \frac{1}{2}|\bu(t)|_{L^2}^2 + \nu \int_{t'}^t \|\bu(s)\|_{H^1}^2 \;\rd s
            \leq \frac{1}{2}|\bu(t')|_{L^2}^2 + \int_{t'}^t (\bbf(s),\bu(s))_{L^2}\;\rd s,
      \end{equation}     
      for all $t$ in $I$ with $t>t'$. These times $t'$ are characterized as the 
      points of strong continuity from the right for $\bu$, and their set is of
      total measure.
    \item \label{lhinitialcont} If $I$ is closed and bounded on the left, with its left
      end point denoted by $t_0$, then the solution
      is strongly continuous in $H$ at $t_0$ from the right, i.e.
      $\bu(t)\rightarrow \bu(t_0)$ in $H$ as $t\rightarrow t_0^+$.
  \end{enumerate}
\end{defs}

From now on, for notational simplicity, \emph{a weak solution will always
mean a Leray-Hopf weak solution.} 

Notice that condition \eqref{lhutinvprime} is actually a consequence of \eqref{lhuinhv} 
and \eqref{lhfnse}.
Condition \eqref{lheineq} on the energy inequality can be interchanged with the 
assumption that $\bu$ satisfies the following energy inequality in the distribution sense
on $I$:
\begin{equation}
  \label{energyinequality}
  \frac{1}{2}\frac{\rd}{\rd t}|\bu(t)|_{L^2}^2 + \nu \|\bu(t)\|_{H^1}^2
    \leq (\bbf(t),\bu(t))_{L^2}.
\end{equation}

The allowed times $t'$ in \eqref{lheineq} 
can also be characterized as the Lebesgue points of the 
function $t\mapsto |\bu(t)|_{L^2}^2$, in the sense that
\begin{equation}
  \label{lebesguepoints}
  \lim_{\tau\rightarrow 0^+} \frac{1}{\tau} \int_{t'}^{t'+\tau} |\bu(t)|_{L^2}^2 \;\rd t
      = |\bu(t')|_{L^2}^2.
\end{equation}
Since $t\mapsto |\bu(t)|_{L^2}^2$ is locally integrable, these Lebesgue points
form a set of full measure. 

Since $\bu$ belongs to $L_\loc^2(I;V)$, condition \eqref{lheineq} implies, upon
use of the Cauchy-Schwarz and Poincar\'e inequalities,
\begin{equation}
   \label{L2enstrophyestimatefinh}
   |\bu(t)|_{L^2}^2 + \nu \int_{t'}^t \|\bu(s)\|_{H^1}^2 \;\rd s
     \leq |\bu(t')|_{L^2}^2 + \frac{1}{\nu\lambda_1} \|\bbf\|_{L^\infty(t',t;H)}^2 (t-t'),
\end{equation}    
for $t'$ and $t$ as in \eqref{lheineq}.

It has been proved in \cite{dascaliuc} that weak solutions 
satisfy a strengthened form of the energy inequality. More precisely,
given a weak solution $\bu$ on an interval $I$, and denoting by $I'$ 
the set of full measure in $I$ defined by the points of strong continuity of
$\bu$ from the right (see the condition \eqref{lheineq} in the Definition 
\ref{deflerayhopfweaksolution}), it follows that for any absolutely continuous, 
nonnegative, nondecreasing function $\psi:[0,\infty)\rightarrow \RR$ with $\psi'$ 
essentially bounded, the solution $\bu$ satisfies
\begin{multline}
  \label{strengthenedenergyineq}
  \frac{1}{2} \psi(|\bu(t)|_{L^2}^2) + \nu \int_{t'}^t \psi'(|\bu(s)|_{L^2}^2)\|\bu(s)\|_{H^1}^2 
    \;\rd s \\
      \leq \frac{1}{2}\psi(|\bu(t')|_{L^2}^2) 
        + \int_{t'}^t \psi'(|\bu(s)|_{L^2}^2)(\bbf(s),\bu(s))_{L^2} \;\rd s,
\end{multline}
for all $t'$ in $I'$ and for all $t$ in $I$ with $t>t'$. The proof in \cite{dascaliuc} has been given in the case of periodic boundary conditions and for $\bbf\in L^2(I;H)$ with $I=[0,T]$, but it can be easily adapted to the case of a bounded domain in $\RR^3$ with a smooth boundary and with more general forces, on an arbitrary interval $I$, such as the case $\bbf\in L^\infty(I,H)$ that we consider.

By using an appropriate sequence of test functions in the inequality \eqref{energyinequality} 
(see \cite[Appendix II.B.1]{fmrt2001a} for the details or \cite[Proposition 7.3]{ball} 
for a different proof), one deduces that a weak solution on an arbitrary interval $I$ also satisfies
\begin{equation}
  \label{energyestimate}
  |\bu(t)|_{L^2}^2 \leq |\bu(t')|_{L^2}^2 e^{-\nu\lambda_1 (t-t')}
     + \frac{1}{\nu^2\lambda_1^2} \|\bbf\|_{L^\infty(t',t;H)}^2
         \left(1-e^{-\nu\lambda_1 (t-t')}\right),
\end{equation}
for almost all $t'$ in $I$ and all $t$ in $I$ with $t'<t$. The allowed times
$t'$ are again the points at which the solution is strongly continuous from the right.

In the case of a weak solution on an interval $I$ closed and bounded on the left
with left end point $t_0$, condition \eqref{lhinitialcont} implies that the point $t_0$ 
is a point of strong continuity from the right, hence the estimate \eqref{energyestimate} 
is also valid for the initial time $t'=t_0$.

Let $R_0$ be given by
\begin{equation}
  \label{defR0}
  R_0 = \frac{1}{\nu\lambda_1}\|\bbf\|_{L^\infty(I;H)}. 
\end{equation}
The energy estimate \eqref{energyestimate} implies the following invariance
property for any ball of radius larger than $R_0$: If $\bu$ is a weak solution
on $[t_0,\infty)$ and $R\geq R_0$, then
\begin{equation}
  \label{invarianceBR}
  \bu(t_0)\in B_H(R) \Rightarrow \bu(t)\in B_H(R), \;\forall t\geq t_0.
\end{equation}

It is well known that given any initial time $t_0$ and any initial
condition $\bu_0$ in $H$, there exists at least one global
weak solution on $[t_0,\infty)$ satisfying $\bu(t_0)=\bu_0$.
See, for instance, \cite{constantinfoias,fmrt2001a,lady63,temam}.

\subsection{Properties of weak solutions}
\label{secpropweak}

In this section we address some further properties of weak solutions. First, in order to
study some topological properties of the space of weak solutions (see Section
\ref{trajectoryspaces}) we will need the following two results.

\begin{lem} [Pasting Lemma]
  \label{pastinglemma}
  Let $\bu^{(1)}$ be a weak solution on an interval $(t_1,t_2]$ and
  $\bu^{(2)}$ be a weak solution on an interval $[t_2,t_3)$, with
  $-\infty\leq t_1<t_2<t_3\leq \infty$ and $\bu^{(1)}(t_2)=\bu^{(2)}(t_2)$.
  Then the function
  \begin{equation}
    \label{concatenation}
      \tilde\bu(t) = \begin{cases}
        \bu^{(1)}(t), & t_1<t< t_2, \\
        \bu^{(2)}(t), & t_2 \leq t < t_3,
    \end{cases}
  \end{equation}
  is a weak solution on the interval $I=(t_1,t_3)$. 
\end{lem}
\medskip

\begin{lem}[Compactness Lemma]
  \label{convergenceofsolutions}
  Let $\{\bu_j\}_{j\in\NN}$ be a sequence of weak solutions on
  some interval $I=(t_0,t_1)$, $-\infty \leq t_0 < t_1 \leq \infty$,
  and suppose this sequence is uniformly bounded in $H$.
  Then, there exists a subsequence $\{\bu_{j'}\}_{j'}$ and a weak
  solution $\bu$ on $I$ such that $\bu_{j'}$ converges to $\bu$
  in $H_\rw$ uniformly on any compact interval in $I$. 
\end{lem}

The proof of Lemma \ref{pastinglemma} is simple (see \cite{frt2010c} for some details). 
The proof of Lemma \ref{convergenceofsolutions} is classical and follows from
uniform (in $j$) estimates of the type \eqref{lhuinhv} and \eqref{lhutinvprime}
of the Definition \ref{deflerayhopfweaksolution} and the Arzela-Ascoli Theorem 
(see e.g. \cite{constantinfoias,lady63,temam,temam3}).

Another result which will be useful is the following criterion
for strong continuity from the right for a weak solution.

\renewcommand{\theenumi}{\arabic{enumi}}
\begin{lem}
  \label{criterionstrongcontinuityright}
  Let $\bu$ be a weak solution on an interval $I\subset \RR$. 
  Let $t'\in I$. Then, the following are equivalent
  \begin{enumerate}
    \item \label{liminfcrit} 
      $\displaystyle \liminf_{\tau\rightarrow 0^+} \frac{1}{\tau}\int_{t'}^{t'+\tau}
      |\bu(s)|_{L^2}^2\;\rd s \leq |\bu(t')|_{L^2}^2$;
    \item \label{strongcontcrit} $\bu$ is strongly continous from the right at $t'$; and
    \item \label{limcrit} 
      $\displaystyle \lim_{\tau\rightarrow 0^+} \frac{1}{\tau}\int_{t'}^{t'+\tau} 
        |\bu(s)|_{L^2}^2\;\rd s = |\bu(t')|_{L^2}^2$;
  \end{enumerate}
\end{lem}

\begin{proof} 
Assume \eqref{liminfcrit} holds. Let $\varepsilon_n$ be a sequence of positive numbers
with $\varepsilon_n\rightarrow 0$. Then, there exists a decreasing sequence of positive 
times $\tau_n \rightarrow 0$ such that 
\[ \frac{1}{\tau_n}\int_{t'}^{t'+\tau_n} \left( |\bu(s)|_{L^2}^2 - |\bu(t')|_{L^2}^2\right)\;\rd s
     \leq \varepsilon_n.
\]
Hence, for each $n\in \NN$, there exists a set $I_n\subset (t',t'+\tau_n)$
of positive Lebesgue measure and such that
\[   |\bu(s)|_{L^2}^2 - |\bu(t')|_{L^2}^2 \leq \varepsilon_n, \qquad
     \text{for all } s\in I_n.
\]
Since $I_n$ is of positive Lebesgue measure, we can
find, for each $n\in \NN$, a time $t_n'\in I_n$ which 
is a point of strong continuity of $\bu$ from the right.
Hence, the following energy inequality holds:
\[  \frac{1}{2} |\bu(t)|_{L^2}^2 + \nu \int_{t_n'}^t \|\bu(s)\|_{H^1}^2 \;\rd s
      \leq \frac{1}{2}|\bu(t_n')|_{L^2}^2 + \int_{t_n'}^t (\bbf(s),\bu(s))_{L^2}\;\rd s,
\]
for all $t\in I$, $t>t_n'$. Moreover, since $t_n'\in I_n$, we 
have that
\[  |\bu(t_n')|_{L^2}^2 - |\bu(t')|_{L^2}^2 \leq \varepsilon_n \rightarrow 0,
\]
as $n\rightarrow \infty$, which means that
\[ \limsup_{n\rightarrow\infty} |\bu(t_n')|_{L^2}^2 \leq |\bu(t')|_{L^2}^2.
\]
This, together with the weak continuity of $\bu$ at $t=t'$, implies
that $\bu(t_n')$ converges strongly in $H$ to $\bu(t')$. Then, 
passing to the limit in the energy inequality, we find that
\[  \frac{1}{2} |\bu(t)|_{L^2}^2 + \nu \int_{t'}^t \|\bu(s)\|_{H^1}^2 \;\rd s
      \leq \frac{1}{2}|\bu(t')|_{L^2}^2 + \int_{t'}^t (\bbf(s),\bu(s))_{L^2}\;\rd s,
\]
for all $t\in I$, $t>t'$. This implies that $t=t'$ is a point of strong continuity
of $\bu$ from the right, which proves \eqref{strongcontcrit}.

Assuming now that $\bu(t)$ converges strongly in $H$ to $\bu(t')$, then 
$t\mapsto |\bu(t)|_{L^2}^2-|\bu(t')|_{L^2}^2$ is continuous
at $t=t'$ and vanishes at $t=t'$, which in turn implies \eqref{limcrit}.
Now it is clear that \eqref{limcrit} implies \eqref{liminfcrit}, which completes the proof.
\end{proof}
\medskip

\subsection{Elements of measure theory}
\label{measuretheory}

In this section, we recall a number of basic results in functional analysis in order
to help the reader whose background is in a more applied area, such as statistical
fluid dynamics.

The statistical solutions defined in this work are Borel
probability measures in some appropriate topological spaces. 
In particular, statistical solutions are defined as families of
Borel probability measures in the phase space of the system. We also consider Borel probability measures in time-dependent function spaces. We shall therefore
collect a number of basic facts from measure theory which are needed 
in this and other works on this subject. For the results mentioned here, 
the reader is referred to \cite{aliprantisborder,bourbaki69,brownpearcy,
rudin,moschovakis,kuratowski}.

In what follows, we will assume that a \emph{topological space} is a \emph{Hausdorff space}, in which any pair of distinct singletons can
be separated by disjoint open sets. 

We recall that the family of \emph{Borel sets} in a given topological space $X$
is the smallest $\sigma$-algebra containing the open sets in that topological
space. We denote the collection of Borel sets in $X$ by $\Bcal(X)$. If the topological space is a reflexive Banach space then 
the Borel sets in the strong topology coincide with the Borel
sets in the weak topology. In particular the Borel sets in $H$
coincide with the Borel sets in $H_\rw$. In this particular case,
Borel sets in $V$ are also
Borel sets in $H$ (see \cite[Appendix IV.A.1, p. 213]{fmrt2001a}).

A \emph{measurable space} is a pair $(X,\Mcal)$ where $X$ is a set and $\Mcal$ is a 
$\sigma$-algebra of subsets of $X$ called the \emph{measurable sets}. 
A \emph{measure space} is a triple $(X,\Mcal,\mu)$ where $(X,\Mcal)$ is a measurable
space and $\mu$ is a measure defined on the sets in $\Mcal$. A probability space is a 
measure space $(X,\Mcal,\mu)$ for which $\mu$ is positive and $\mu(X)=1$. A measure is said to be \emph{complete} if any subset of
a \emph{null set} (a measurable set with measure zero) is also measurable.

Given two measurable spaces $(X,\Mcal)$ and $(Y,\Ncal)$ and a function
$f:X\rightarrow Y$, the function $f$ is said to be measurable, or $(\Mcal,\Ncal)$-measurable,
if $f^{-1}(E)\in \Mcal$ for all $E\in \Ncal$.

A \emph{Borel measure} on a topological space $X$ is a measure $\mu$ on $X$  
defined on the $\sigma$-algebra of the Borel sets $\Bcal(X)$ of $X$. 
When the target space is metrizable (and the measure in the target space is the Borel
measure) then the pointwise limit of a sequence of measurable functions from a 
measurable space into that metrizable space is measurable
(see \cite[Lemma 4.29]{aliprantisborder}).

Given a Borel measure $\mu$ on a topological space $X$,
the $\sigma$-algebra $\Bcal_\mu(X)$ is defined as the smallest $\sigma$-algebra
containing the Borel sets and the subsets of Borel sets of $\mu$-measure zero. 
One can show that
$E\in\Bcal_\mu(X)$ if and only if there exists a Borel set $E_B$ and a subset 
$E_N$ of a Borel set of $\mu$-measure zero such that $E=E_B\cup E_N$.
This representation of $E\in\Bcal_\mu(X)$ may not be unique but the 
$\mu$-measure
of $E_B$ is independent of the representation, so that we can extend the Borel
measure $\mu$ to a complete Borel measure on $\Bcal_\mu(X)$ by defining
$\mu(E)=\mu(E_B)$. Such a measure is called the \emph{Lebesgue extension}
of the Borel measure and is still denoted by $\mu$.
Moreover, we call the elements in $\Bcal_\mu(X)$ as $\mu$-measurable sets.
The collection $\Bcal_\mu(X)$ of $\mu$-measurable sets is usually larger than the 
collection of Borel sets $\Bcal(X)$.

Given two topological spaces $X$ and $Y$ and a continuous function $f:X\rightarrow Y$,
it follows that $f$ is a \emph{Borel map} in the sense that $f^{-1}(E)\in \Bcal(X)$ for all 
$E\in\Bcal(Y)$. Such a continuous function is also $(\Bcal_\mu(X),\Bcal(Y))$-measurable,
with respect to the extension of a Borel measure $\mu$ on $X$.
Notice, however, that given two (Lebesgue extensions of) Borel measures
$\mu$ and $\nu$ on $X$ and $Y$, respectively, a continuous function $f$ may not
be \emph{measurable} from $(X,\Bcal_\mu(X),\mu)$ to $(Y,\Bcal_\nu(Y),\nu)$ since
$f^{-1}(E)$ may not belong to $\Bcal_\mu(X)$ for all $E$ in $\Bcal_\nu(Y)$ (just
take $f$ to be the identity to get a contradiction).
  
A \emph{carrier} of a measure is any measurable subset of \emph{full measure}, i.e.
its complement is of zero measure. The 
\emph{support} of a Borel measure is the smallest closed set of full measure.

Given a topological space $X$, we denote by $\Ccal_\rb(X)$ the space of bounded
real-valued continuous functions on $X$, and by $\Ccal_\rc(X)$ the space of
compactly supported real-valued continuous functions on $X$.

A \emph{locally compact Hausdorff space} $X$ is a (Hausdorff) topological space
with the property that every point has a compact neighborhood.
The Kakutani-Riesz Representation theorem \cite{bourbaki69,rudin} 
asserts that a positive linear functional $f$ defined on the space of 
compactly supported continuous real-valued functions $\Ccal_\rc(X)$ 
on a locally compact Hausdorff space $X$ can be uniquely 
represented by a Borel measure 
$\mu$ on $X$, with 
\[ f(\varphi)=\int_X \varphi(\bu)\;\rd\mu(\bx), \quad \text{for all } \varphi \in \Ccal_\rc(X).
\]

Given two measurable spaces $(X,\Mcal)$ and $(Y,\Ncal)$, 
a measurable function $T:X\rightarrow Y$,
and a probability measure $\mu$ on $(X,\Mcal)$, one can define a probability 
measure $\mu_T$ on $Y$ by the formula $\mu_T(E)=\mu(T^{-1}(E))$, for all measurable 
subsets $E$ of $Y$. The measure $\mu_T$ is called the \emph{measure induced} from 
$\mu$ by $T$, and is sometimes denoted $T\mu$ or $\mu T^{-1}$ 
(see \cite[Section 13.12]{aliprantisborder}). 
It follows that
\begin{equation}
  \label{changeofvariablesinducedmeasures}
  \int_Y \varphi(y) \;\rd\mu_T(y) = \int_X \varphi(T(x))\;\rd\mu(x), \qquad \forall \varphi\in L^1(\mu).
\end{equation}
Moreover, 
\begin{equation}
  \label{integrabilityinducedmeasures}
  \text{$\psi$ is $\mu_T$-measurable and $\psi\circ T\in L^1(\mu)$ if and 
    only if $\psi\in L^1(\mu_T)$.}
\end{equation}

If $\mu$ is a regular Borel measure (as defined in 
\eqref{upperregmeas}-\eqref{compactsetscontainedinE} below) 
on a locally compact Hausdorff space $X$, then
$\Ccal_\rc(X)$ is dense in $L^p(\mu)$, for $1\leq p <\infty$ \cite[Theorem 13.9]{aliprantisborder},
and \eqref{changeofvariablesinducedmeasures} holds for all $\varphi\in \Ccal_\rc(X)$.

In the case $X$ and $Y$ are locally compact topological spaces, a continuous
map $T:X\rightarrow Y$ induces an operator $L_T:\Ccal_\rc(Y)\rightarrow \Ccal_\rc(X)$
given by $L_T\varphi = \varphi\circ T$. Then, regarding a Borel probability measure $\mu$ 
on $X$ as an element of the dual space $\Ccal_\rc(X)'$, it is natural to view the 
induced measure $T\mu$ as $L_T^*\mu$, where 
$L_T^*:\Ccal_\rc(X)'\rightarrow \Ccal_\rc(Y)'$
is the adjoint of the operator $L_T$.

In a metrizable topological space $X$, the following statements concerning
two Borel probability measures $\mu$ and $\nu$ are equivalent \cite[Theorem 15.1]{aliprantisborder}:
\begin{align*}
   \mu=\nu 
   \Leftrightarrow & \mu(G)=\nu(G) \text{ for all open sets } G \\
   \Leftrightarrow & \mu(F)=\nu(F) \text{ for all closed sets } F \\
   \Leftrightarrow & \int_X \varphi(x)\;\rd\mu(x) = \int_X \varphi(x)\;\rd\nu(x), 
     \forall \varphi\in \Ccal_\rb(X) \\
   \Leftrightarrow & \int_X \varphi(x)\;\rd\mu(x) = \int_X \varphi(x)\;\rd\nu(x), \forall 
     \varphi\in \Dcal,
\end{align*}
where $\Dcal$ is any dense subset of $\Ccal_\rb(X)$

In a metrizable topological space $X$ we say that a net $\{\mu_\alpha\}_\alpha$ 
of Borel probability measures on $X$ converges weak-star to a Borel probability measure
$\mu$ on $X$ if 
\[ \int_X \varphi(x) \;\rd\mu_\alpha(x) \rightarrow \int_X \varphi(x) \;\rd\mu(x), 
  \;\forall \varphi\in\Ccal_\rb(X).
\]
(See \cite[Section 15.1]{aliprantisborder}.) This convergence is denoted by
\[ \mu_\alpha \stackrel{*}{\rightharpoonup} \mu.
\]

A Borel measure $\mu$ on a topological space $X$ is called \emph{regular} when
\begin{align}
  \label{upperregmeas}
  \mu(E) & = \inf\{\mu(O); \;O\supset  E, \;O \text{ open in } X\}
    = \inf_{O\in\mathcal{O}_X(E)}\mu(O), \text{ and} \\
  \label{lowerregmeas}
  \mu(E) & = \sup\{\mu(K); \;K\subset  E, \;K \text{ compact in } X\}
    = \sup_{K\in\mathcal{K}_X(E)}\mu(K),
\end{align}
where
\begin{align}
  \label{opensetscontainingE}
  \mathcal{O}_X(E) & = \text{ family of open sets in $X$ containing } E, \\
  \label{compactsetscontainedinE}
  \mathcal{K}_X(E) & = \text{ family of compact sets in $X$ contained in } E.
\end{align}
The first relation is called \emph{upper regularity} and the second one,
\emph{lower regularity}.\footnote{Some authors refer to this notion of lower regularity 
as tightness, and define lower regularity in terms of closed sets. In this sense,
any Borel measure on a metrizable space is regular 
\cite[Theorem 12.5]{aliprantisborder}, but not necessarily tight.} 
The Lebesgue extension of a regular Borel measure is also a regular measure.

Now if $K$ is a compact subspace of a locally convex topological vector space $V$
and $E$ is the set of extremal points of $K$ then the Krein-Milman theorem 
\cite{dunfordschwartz} asserts that the closed convex hull $\overline{\text{co}} E$
of $E$ coincides with that of $K$. In the case $X$ is a compact 
metric space and we take $V=\Ccal(X)'$ and $K$ to be the subset $V$ of positive linear
functionals with norm one, then the weak-star topology in $V$ is metrizable, and $K$ is convex and weakly-star compact and is made of the Borel 
probability measures on $X$. The extremal points $E$ of $K$ are precisely the Dirac measures on $X$. The Krein-Milman theorem applies and yields 
that any Borel probability measure on $X$ is the weak-star limit of convex 
combinations of Dirac measures on $X$.

A space that plays an important role in measure theory
is the \emph{Polish space}, which is a separable and complete metrizable space. 
An important property is that \emph{any finite Borel measure on a Polish space is 
regular in the sense above} \cite[Theorem 12.7]{aliprantisborder}. This property will 
be particularly exploited in a forthcoming paper \cite{frtssp2}, which considers 
stationary statistical solutions.

In our case, all the spaces $H$, $V$, $V'$ are Polish,
and so are the bounded, weakly closed subsets of $H$ endowed with the weak
topology, such as $B_H(R)_\rw$, $R>0$.  The space $H_\rw$, however, 
is not Polish. 

\subsection{Time-dependent function spaces}
\label{timedependentfuncionspacessec}

We define some basic ``time-dependent'' function spaces. In what follows, we consider
an arbitrary interval $I$ in $\RR$. First, we consider the 
spaces $\Ccal_\loc(I,H_\rw)$ and $\Ccal_\loc(I,B_H(R)_\rw)$, with $R>0$, endowed with the topology 
of uniform weak convergence on compact intervals in $I$. With this topology, since $H_\rw$
is separable, the space $\Ccal_\loc(I,H_\rw)$ is a separable Hausdorff locally convex topological 
vector space (the proof can be adapted from \cite[Lemma 3.99]{aliprantisborder} using that
the topology is that of uniform convergence on compact subintervals, and using that
bounded sets in $H_\rw$ are metrizable), and $\Ccal_\loc(I,B_H(R)_\rw)$ is a Polish space. In the case $I$ is compact, we write simply $\Ccal(I,H_\rw)$ and $\Ccal(I,B_H(R)_\rw)$.

The topology of $\Ccal_\loc(I,H_\rw)$ can be characterized by a basis of neighborhoods of 
the origin, given by
\begin{equation}
   \Ocal(J,O_\rw) = \left\{\bv\in \Ccal_\loc(I,H_\rw); \; \bv(t)\in O_\rw, \;\forall t\in J\right\},
\end{equation}
where $J$ is a compact subset of $I$ (or just a compact interval in $I$, the generated
topology is the same) and $O_\rw$ is a (weak) neighborhood of the origin in $H_\rw$.

For intervals $J\subset I \subset \RR$, we define the restriction operator given by
\begin{equation}
  \begin{aligned}
    \Pi_J :  \Ccal_\loc(I, H_\rw) & \: \rightarrow \;\Ccal_\loc(J, H_\rw) \\
       \bu & \;\mapsto \;(\Pi_J \bu)(t) = \bu(t), \;\forall t\in J.
  \end{aligned}
\end{equation}
For the sake of notational simplicity, we do not make explicit the dependence
of the operator on $I$. This should be clear in the context. 
These operators are continuous. In case $J$ is closed in $I$, 
$\Pi_J$ is also surjective and open, in the sense of taking an 
open set in $\Ccal_\loc(I,H_\rw)$ into an open set in $\Ccal_\loc(J,H_\rw)$. 

For each interval $I\subset \RR$ and each $t\in I$, we also define the projection operators
\begin{equation}
  \begin{aligned}
    \Pi_t : \Ccal_\loc(I, H_\rw) & \;\rightarrow \;H_\rw \\
       \bu & \;\mapsto \;\Pi_t \bu = \bu(t),
  \end{aligned}
\end{equation}
which are also continuous and open, as well as surjective.

These operators are crucial for passing from trajectory to phase space and are
also important for the characterization of certain spaces
and for proving some topological properties for them.

For example, given $\bu\in \Ccal_\loc(I,H_\rw)$, we have that $\bu$ is bounded on each compact 
interval $J\subset I$. Hence, given any sequence of numbers $R_k> 0$, 
$R_k\rightarrow \infty$,
and any sequence of compact intervals $J_n\Subset I$, $I=\cup_n J_n$, we can write 
\[ \Ccal_\loc(I,H_\rw)=\bigcap_n \bigcup_k \Pi_{J_n}^{-1}\Ccal(J_n,B_H(R_k)_\rw).
\]

The bounded sets in $\Ccal_\loc(I,H_\rw)$ can be characterized as the sets $\Bcal$
for which there exists an increasing sequence of compact intervals $J_n\Subset I$ with
$I=\cup_n J_n$, and an increasing sequence of numbers $R_n>0$ with 
$R_n\rightarrow \infty$, such that 
\[ \Bcal \subset \bigcap_n \Pi_{J_n}^{-1}\Ccal(J_n,B_H(R_n)).
\]
Any closed bounded set $\Bcal$ endowed with the topology inherited from $\Ccal_\loc(I,H_\rw)$ 
is a Polish space, hence $\Ccal_\loc(I,H_\rw)$ can be termed a ``quasi-Polish'' space
(we call a \emph{quasi-Polish space} any topological vector space such that any closed 
bounded subset is Polish).

The compact subsets of $\Ccal_\loc(I,H_\rw)$ can be characterized as the sets
$\Kcal$ for which for every compact interval $J\subset I$, the subset $\Pi_J\Kcal$ 
is equi-bounded with respect to the norm of $H$, i.e. there
exists $R>0$ such that $|\bu(t)|_{L^2}\leq R$ for all $\bu\in \Kcal$ and all $t\in J$, and 
$\Pi_J\Kcal$ is equicontinuous with respect to the uniform structure of 
$\Ccal(J,H_\rw)$, i.e. for any (weakly) open neighborhood $O_\rw$ of the origin in $H_\rw$,
there exists $\delta>0$ such that $\bu(t_2)-\bu(t_1)\in O_\rw$ for every $t_1,t_2\in J$
with $|t_2-t_1|<\delta$.

\subsection{Trajectory spaces}
\label{trajectoryspaces}

We now consider some spaces of weak solutions and study their topological properties. 
In what follows, we consider $R\geq R_0$, where $R_0$ is given by \eqref{defR0}, 
and intervals $I\subset \RR$, and denote by 
$I^\circ$ the interior of $I$. 

We define the spaces
\begin{align}
  \Ucal_I
      & = \{ \bu\in \Ccal_\loc(I,H_\rw); \;\bu \text{ is a weak solution on } I\}, \\
  \Ucal_I(R)
      & = \left\{ \bu\in \Ccal_\loc(I,B_H(R)_\rw); \;\bu \text{ is a weak solution on } I
         \right\}, \\
  \Ucal_I^\sharp
      & = \{ \bu\in \Ccal_\loc(I,H_\rw); \;\bu \text{ is a weak solution on } I^\circ\}, \\
  \Ucal_I^\sharp(R)
      & = \left\{ \bu\in \Ccal_\loc(I,B_H(R)_\rw); \;\bu \text{ is a weak solution on } 
         I^\circ \right\},
\end{align}
endowed with the topology inherited from $\Ccal_\loc(I,H_\rw)$.

\begin{rmk}
It is straighforward to see that $\Ucal_I\subset\Ucal_I^\sharp$
and $\Ucal_I(R)\subset\Ucal_I^\sharp(R)$. When
$I$ is open on the left, then actually $\Ucal_I=\Ucal_I^\sharp$ 
and $\Ucal_I(R)=\Ucal_I^\sharp(R)$. 
The main difference between the spaces $\Ucal_I$ and $\Ucal_I^\sharp$ 
(and between $\Ucal_I(R)$ and $\Ucal_I^\sharp(R)$) is in the case in which 
$I$ is closed and bounded on the left, for which the solutions in $\Ucal_I$ 
are necessarily strongly continuous from the right at the left end point of $I$,
while the solutions in $\Ucal_I^\sharp$ are weakly continuous but not necessarily
strongly continuous. The space $\Ucal_I^\sharp$ is in fact the sequential
closure of $\Ucal_I$. In the bounded case, since $\Ucal_I^\sharp(R)$ is metrizable and
complete (see Proposition \ref{Ucalspacesprop}), we have that this space is in fact the 
closure of $\Ucal_I(R)$. 

In order to illustrate that the inclusions could be strict when $I$ is closed and bounded on 
the left, or at least that we cannot prove equality, consider the case in which $I=[0,T)$. 
Let $\bu=\bu_n$ be a sequence of weak solutions of \eqref{nseeq} on $I=[0,T)$ with
initial data $\bu(0)=\bu_{0n}$. Assume that,
as $n\rightarrow \infty$, $\bu_{0n}$
converges weakly to $\bu_0$ in $H$. Then, it is not difficult to deduce (similarly to
Lemma \ref{convergenceofsolutions} or see the existence results in 
\cite{constantinfoias,fmrt2001a,lady63,temam,temam3}) that a 
subsequence of $\bu_n$ converges to a limit $\bu$ in $L^2(0,T;V)$ weakly and
in $L^\infty(0,T;H)$ weak-star and that $\bu$ is a weak solution on $(0,T)$
and not necessarily on $[0,T)$; that is we are not able to prove 
\eqref{energyinequalityintegralform} for $t'=0$. The sequence $\{\bu_n\}_n$ belongs
to $\Ucal_I$ but we can only guarantee that the limit point $\bu$ belongs to
$\Ucal_I^\sharp$.
\qed
\end{rmk}

We start with the following result.

\begin{lem}
  Let $I\subset \RR$ be an arbitrary interval and $R\geq R_0$.
  Then, $\Ucal_I(R)$ and $\Ucal_I^\sharp(R)$ are not empty.
\end{lem}

\begin{proof}
  First, since $\Ucal_I(R)\subset\Ucal_I^\sharp(R)$, it suffices to show that $\Ucal_I(R)$
  is not empty. This amounts to showing that there exists a Leray-Hopf weak solution
  on $I$ with $\bu(t)\in B_H(R)$ for all $t\in I$. If $I$ is closed and bounded on the left
  with left end point $t_0$, then we take an arbitrary initial condition $\bu_0\in B_H(R)$,
  use the well-known existence result (see Section \ref{NSEsettingsec}) to obtain a global 
  Leray-Hopf weak solution $\bu$ on $[t_0,\infty)$, use the energy equation and,
  in particular, the invariance property \eqref{invarianceBR} to deduce that $\bu(t)\in B_H(R)$
  for all $t\geq t_0$ and finally restrict $\bu$ to $I$ to have an element in $\Ucal_I(R)$.
  
  If $I$ is open on the left with left end point $t_0\geq -\infty$, then we take a sequence of 
  initial times $t_{0n}\rightarrow t_0$ and take a fixed 
  initial condition $\bu_0\in B_H(R)$ (or even a sequence of initial conditions provided
  they are all within $B_H(R)$). Then, we construct a sequence of weak solutions $\bu_n$,
  each of them in $\Ucal_{[t_{0n},\infty)}(R)$. Using Lemma \ref{convergenceofsolutions}
  and a diagonalization process, we find a subsequence converging to a weak solution 
  $\bu$ on $(t_0,\infty)$, with $\bu(t)\in B_H(R)$, for all $t>t_0$.
  Restricting this solution to $I$ yields an element in $\Ucal_I(R)$.
\end{proof}

\renewcommand{\theenumi}{\arabic{enumi}}
\begin{prop}
  \label{Ucalspacesprop}
  Let $I\subset \RR$ be an arbitrary interval and let $R\geq R_0$. Then,
  \begin{enumerate}
    \item \label{uandusharplctvs} 
      The spaces $\Ucal_I$ and $\Ucal_I^\sharp$ are separable Hausdorff 
      spaces;
    \item \label{urseparablemetrizable} 
      The space $\Ucal_I(R)$ is a separable metrizable space;
    \item \label{uisharppolish} The space $\Ucal_I^\sharp(R)$ is a Polish space.
  \end{enumerate}
\end{prop}

\begin{proof}
  Properties \eqref{uandusharplctvs} and \eqref{urseparablemetrizable} 
  are not difficult to check. The only step which is not immediate 
  is the completeness of $\Ucal_I^\sharp(R)$. Since this space is metrizable, it suffices
  to check sequential completeness. The sequential completeness follows
  then from Lemma \ref{convergenceofsolutions}.
\end{proof}
\medskip

The Leray-Hopf weak solutions belong to $\Ucal_I$, 
so this is the natural space to consider, but the larger space $\Ucal_I^\sharp$ is 
needed because each space $\Ucal_I^\sharp(R)$ is compact as we will prove 
below\footnote{In \cite{fmrt2001a} 
we said that $\Ucal_I(R)$ is complete, which might not be true. Therefore, 
in some proofs in \cite{fmrt2001a}, one has to replace $\Ucal_I(R)$ with 
$\Ucal_I^\sharp(R)$, as it is done in the present paper.}.
Nevertheless, we show that $\Ucal_I$ and $\Ucal_I(R)$ are at least Borel subsets 
of $\Ccal_\loc(I,H_\rw)$ and $\Ccal_\loc(I,B_H(R)_\rw)$, respectively. 
  
\begin{prop}
  \label{ucalitildercompact}
  Let $I\subset\RR$ be an arbitrary interval and let $R\geq R_0$. Then
  $\Ucal_I^\sharp(R)$ is compact in $\Ccal_\loc(I,B_H(R)_\rw)$ and, hence, it is 
  a compact metric space, and it is compactly embedded
  in $L_\loc^2(I;H)$.
\end{prop}  

\begin{proof}
  The classical a~priori estimates for the weak solutions
  (see e.g. \cite{constantinfoias,fmrt2001a,lady63,temam}) yield that
  \begin{align*}
    \text{(1)} & \text{ a bounded subset of } \Ucal_I^\sharp
        \text{ is bounded in } L_\loc^2(I;V), \\
     \text{(2)} & \left\{ \ddt{\bu}; \;\bu\in \text{bounded subset of } \Ucal_I^\sharp
        \right\} \text{ is bounded in } L_\loc^{4/3}(I;V').
  \end{align*}
  Therefore, $\Ucal_I^\sharp(R)$ is bounded in $\Lcal_I$, where
  \[ \Lcal_I = \{ \bu \in L_\loc^2(I; V); \;\bu'\in L_\loc^{4/3}(I;V')\},
  \]
  endowed with the natural metric. Notice that $\Lcal_I$ is a separable 
  Fr\'echet space, hence also a Polish space. Moreover, by the
  Aubin Compactness Theorem, $\Lcal_I$ is compactly embedded in
  $L^2_\loc(I;H)$. Therefore, $\Ucal_I^\sharp(R)$ is relatively compact in $L_\loc^2(I;H)$.

  By the Arzela-Ascoli Theorem, the sets $\Bcal = \{\bu\in \Bcal_1; \bu'\in \Bcal_2\}$,
  where $\Bcal_1$ is any bounded set in $\Ccal_\loc(I;H_\rw)$ and $\Bcal_2$ is
  any bounded set in $L_\loc^{4/3}(I;V')$, are relatively compact in $\Ccal_\loc(I;H_\rw)$.
  From this we infer that $\Ucal_I^\sharp(R)$ is relatively compact in $\Ccal_\loc(I;H_\rw)$.

  Since  $\Ucal_I^\sharp(R)$ is closed and included in $\Ccal_\loc(I,B_H(R)_\rw)$, we deduce that 
  $\Ucal_I^\sharp(R)$ is compact in $\Ccal_\loc(I,B_H(R)_\rw)$ and, hence, it is a compact 
  metrizable space with respect to the topology inherited from $\Ccal_\loc(I;H_\rw)$. It is also
  compactly embedded in $L_\loc^2(I;H)$.
\end{proof}
\medskip
  
The next result shows, in particular, that, in case $I$ is an interval open on the left, then $\Ucal_I$ and $\Ucal_I^\sharp$ are Borel $\Fcal_{\sigma\delta}$ sets in $\Ccal_\loc(I,H_\rw)$, i.e. they are countable intersections of countable unions of closed sets in $\Ccal_\loc(I,H_\rw)$.
  
\begin{prop}
  Let $I\subset \RR$ be an interval open on the left. Then, for any sequence
  $\{R_k\}_{k\in \NN}$ of positive numbers with $R_k\geq R_0$ and 
  $R_k\rightarrow \infty$ and for any sequence $\{J_n\}_{n\in \NN}$ of compact intervals 
  in $I$ with $I=\cup_n J_n$, we have the characterization
  \begin{equation}
    \label{characucaliopen}
    \Ucal_I =\Ucal_I^\sharp=\bigcap_n \bigcup_k\Pi_{J_n}^{-1}\Ucal_{J_n}^\sharp(R_k).
  \end{equation}
  In particular, $\Ucal_I$ and $\Ucal_I^\sharp$ are Borel $\Fcal_{\sigma\delta}$ sets in 
  $\Ccal_\loc(I,H_\rw)$. 
\end{prop}

\begin{proof}
  Let $\bu\in\Ucal_I=\Ucal_I^\sharp$. Since $\bu$ is weakly continuous, it
  follows from the Banach-Steinhaus theorem that $\bu(t)$ is bounded in $H$
  uniformly for $t$ on any compact interval in $I$. Therefore, for each $n$,
  there exists $k$ sufficiently large such that $\Pi_{J_n}\bu\in \Ucal_{J_n}^\sharp(R_k)$.
  Hence, for each $n$, $\bu$ belongs to $\bigcup_k\Pi_{J_n}^{-1}\Ucal_{J_n}^\sharp(R_k)$.
  This implies \eqref{characucaliopen}.
  
  Since each $\Ucal_I^\sharp(R_k)$ is compact, and $\Pi_J$ is continuous,
  the characterization \eqref{characucaliopen} shows that the spaces are
  Borel $\Fcal_{\sigma\delta}$ sets in $\Ccal_\loc(I,H_\rw)$. 
\end{proof}
\medskip

\begin{prop}
  \label{characbdducal}
  Let $I\subset \RR$ be an interval open on the left. Then, the bounded 
  subsets of $\Ucal_I=\Ucal_I^\sharp$ can be characterized as the sets $\Bcal$
  for which there exists an increasing sequence of compact intervals $J_n\Subset I$ with
  $I=\cup_n J_n$, and an increasing sequence of numbers $R_n\geq R_0$ with 
  $R_n\rightarrow \infty$, 
  such that 
  \[ \Bcal \subset \bigcap_n \Pi_{J_n}^{-1}\Ucal_{J_n}^\sharp(R_n).
  \]
  Moreover, any closed bounded subset of $\Ucal_I=\Ucal_I^\sharp$ is compact, 
  and, thus, $\Ucal_I=\Ucal_I^\sharp$ is quasi-compact, i.e. a space such that every 
  closed bounded subset is compact. 
\end{prop}

\begin{proof}
  The characterization is not difficult to see. The compactness of the closed bounded sets
  follows from the compactness of each set $\Ucal_{J_n}^\sharp(R_n)$ and
  by using a diagonalization process.
\end{proof}
\medskip

\begin{prop}
  \label{ucalisharpsigmacompact}
  Let $I\subset\RR$ be an interval closed and bounded on the left. Then for any sequence
  $\{R_k\}_{k\in \NN}$ of positive numbers with $R_k\geq R_0$ and $R_k\rightarrow \infty$, 
  we have the representation
  \begin{equation}
    \label{characucalitildeclosed}
    \Ucal_I^\sharp = \bigcup_k\Ucal_I^\sharp(R_k).
  \end{equation}
  In particular, $\Ucal_I^\sharp$ is $\sigma$-compact in $\Ccal_\loc(I,H_\rw)$, i.e.
  it is a countable union of compact sets in $\Ccal_\loc(I,H_\rw)$. Moreover, any 
  bounded subset of $\Ucal_I^\sharp$ must be included in $\Ucal_I^\sharp(R_k)$
  for $k$ sufficiently large.
\end{prop}

\begin{proof}
  Let $\bu\in\Ucal_I^\sharp$. Since $\bu$ is weakly continuous on $I$, it follows from the 
  Banach-Steinhauss theorem that $\bu(t)$ is bounded in $H$ for $t$ in any compact interval
  in $I$ containing the left end point of $I$. Then, applying the energy estimate 
  \eqref{energyestimate} starting from any ``good'' point $t'$ in this compact interval, we
  find that $\bu$ is bounded on all $I$. Hence, it must belong to $\Ucal_I^\sharp(R_k)$
  for $k$ sufficiently large. Therefore, \eqref{characucalitildeclosed} holds. Since 
  each $\Ucal_I^\sharp(R_k)$ is compact, we deduce that $\Ucal_I^\sharp$ is
  $\sigma$-compact. The characterization of the bounded sets is easy to see.
\end{proof} 
\medskip

\begin{prop}
  \label{Ucalileftborel}
  Let $I\subset\RR$ be an interval closed and bounded on the left and let $R\geq R_0$. 
  Then, $\Ucal_I$ and $\Ucal_I(R)$ are Borel sets in $\Ccal_\loc(I,H_\rw)$.
  Moreover, for any sequence $\{R_k\}_{k\in \NN}$ of positive numbers 
  with $R_k\rightarrow \infty$, we have the characterization
  \begin{equation}
    \label{characucaliclosed}
    \Ucal_I = \bigcup_k\Ucal_I(R_k).
  \end{equation}
  Furthermore, any bounded subset of $\Ucal_I$ must be included in 
  $\Ucal_I(R_k)$ for $k$ sufficiently large.
\end{prop}

\begin{proof}  
  Let $t_0$ denote the left end point of $I$. Denote by $\Tcal$ the topology
  defined earlier for $\Ccal_\loc(I,B_H(R)_\rw)$ and 
  denote by $\tilde\Tcal$ the topology in $\Ccal_\loc(I,B_H(R)_\rw)$ associated
  with the uniform weak convergence on compact intervals in $I$ \emph{and} 
  strong convergence at $t_0$. 
  More precisely, this topology can be characterized by a system of  
  neighborhoods of the origin given by
  \[ \Ncal(J,O_\rw,\varepsilon) = \{\bv\in\Ccal_\loc(I,B_H(R)_\rw); 
         \;\bv(t)\in O_\rw, \;\forall t\in J, \;|\bv(t_0)|_{L^2}\leq \varepsilon,\},
  \]
  where $J\subset I$ is compact, $O_\rw$ is a neighborhood of the origin
  in $B_H(R)_\rw$, and $\varepsilon>0$. This space is metrizable and, in fact, it is
  a Polish space. 
  
  The topology $\tilde\Tcal$ is finer than $\Tcal$, i.e. $\Tcal$ is contained in $\tilde\Tcal$. 
  We claim that the corresponding Borel sets generated by these two topologies 
  are the same. Since $\tilde\Tcal$ is finer than $\Tcal$, it suffices
  to show that every open set in the former topology is a Borel set in the
  latter topology. For that, it suffices to show that every neighborhood 
  $\Ncal(J,O_\rw,\varepsilon)$
  is a Borel set in $\Ccal_\loc(I,B_H(R)_\rw)$. In fact, $\Ncal(J,O_\rw,\varepsilon)$
  is a $\Gcal_\delta$ set in $\Ccal_\loc(I,B_H(R)_\rw)$ since it can be written as
  \[ \Ncal(J,O_\rw,\varepsilon) = \bigcap_{m\in \NN}\left(\Ocal(J,O_\rw)\cap
     \Ocal(\{t_0\},O^{m,\varepsilon}_\rw)\right),
  \]
  where $O^{m,\varepsilon}_\rw=\{\bv\in H; |P_m\bv|_{L^2}< \varepsilon\}$
  and $P_m$ is the 
  Galerkin projector defined in Section \ref{NSEsettingsec}.
  
  Consider now $\Ucal_I(R)$ with respect to the topology $\tilde\Tcal$.
  As such, it is a closed set since the strong convergence at $t_0$
  together with the energy inequality \eqref{energyestimate} with 
  $t'=t_0$ implies that the limit of a sequence of weak solutions still satisfies 
  the energy inequality with $t'=t_0$, from which we deduce that the 
  limit solution is strongly continuous from the right at $t_0$. Note that
  since the topology $\tilde\Tcal$ in $\Ccal_\loc(I,B_H(R)_\rw)$ is also metrizable, 
  it suffices to consider sequences. 

  Now, since $\Ucal_I(R)$ is closed for $\tilde\Tcal$
  it is Borel with respect to $\tilde\Tcal$, hence it is also Borel
  in $\Ccal_\loc(I,B_H(R)_\rw)$.
  
  Now, due to the energy estimate \eqref{energyestimate}, any
  $\bu\in \Ucal_I$ is uniformly bounded on $I$, hence it belongs
  to $\Ucal_I(R_k)$ for sufficiently large $k$. Thus, the characterization
  \eqref{characucaliclosed} follows and implies that $\Ucal_I$ is also Borel.
  
  The characterization of the bounded sets is not difficult to see. 
\end{proof}
\medskip

\subsection{Further topological properties for some trajectory spaces}

We now show an interesting result saying that $\Ucal_I$ is a $\Gcal_{\delta\sigma}$-set 
within $\Ucal_I^\sharp$ (i.e. $\Ucal_I$ is a countable union of sets which are countable intersections of relatively open sets in $\Ucal_I^\sharp$). This result is a refinement of the
result in Proposition \ref{Ucalileftborel}, but the previous one is also presented because 
of its simplicity. We also show, in the case the forcing term is time-independent,
that $\Ucal_I$ is a large set within $\Ucal_I^\sharp$, in the sense of being also dense
in it. First, let us prove the following lemma. 

\renewcommand{\theenumi}{\arabic{enumi}}
\begin{lem}
  \label{lemphicondition}
  Consider a time interval $I$ which is closed and bounded on the left,
  with left end point $t_0$. Then, the function
  \begin{equation}
    \label{definitionofPsi}
    \Psi(\bu,t) = \frac{1}{t-t_0}\int_{t_0}^t \left(|\bu(s)|_{L^2}^2-|\bu(t_0)|_{L^2}^2\right)\;\rd s
  \end{equation}
  is a Borel function in $(\bu,t)\in \Ccal_\loc(I,H_\rw)\times (I\setminus \{t_0\})$.
  Moreover, for each fixed $\bu\in \Ccal_\loc(I,H_\rw)$, the function
  $t\mapsto \Psi(\bu,t)$ is continuous in $t\in I\setminus \{t_0\}$.
  Furthermore, the function 
  $\bu\mapsto\liminf_{t\rightarrow t_0^+} \Psi(\bu,t)$ is a Borel function on 
  $\Ccal_\loc(I,H_\rw)$ with
  \begin{equation}
    \label{liminfpsigeqzero}
       \liminf_{t\rightarrow t_0^+} \Psi(\bu,t) \geq 0, \text{ for all } \bu\in \Ccal_\loc(I,H_\rw).
  \end{equation} 
  Finally, given $\bu\in\Ucal_I^\sharp$, the following statements are equivalent
  \begin{enumerate}
    \item \label{liminfcritpsi} $ \liminf_{t\rightarrow t_0^+} \Psi(\bu,t) = 0,$ \smallskip
    \item \label{uinucalicrit} $\bu\in \Ucal_I$; and \medskip
    \item \label{limcritpsi} $\lim_{t\rightarrow t_0^+} \Psi(\bu,t) = 0$.
  \end{enumerate}
\end{lem}

\begin{proof} 
For $m\in \NN$, let $P_m$ be the Galerkin projector onto the first $m$ modes
of the Stokes operator. Consider the functions
\[ \Psi_m(\bu,t) = \Psi(P_m\bu,t) 
    = \frac{1}{t-t_0} \int_{t_0}^t \left( |P_m\bu(s)|_{L^2}^2 - |P_m\bu(t_0)|_{L^2}^2\right) \;\rd s,
\]
which are clearly continuous functions on $\Ccal_\loc(I,H_\rw)\times (I\setminus\{t_0\})$.
Since $\bu$ is weakly continuous on 
$I$, the function
\[ t \mapsto |P_m\bu(t)|_{L^2}^2 - |P_m\bu(t_0)|_{L^2}^2
\]
is uniformly bounded on $[t_0,t]$ and converges pointwise to the
function
\[ t \mapsto |\bu(t)|_{L^2}^2 - |\bu(t_0)|_{L^2}^2,
\]
as $m\rightarrow \infty$. Therefore, by the Lebesgue Dominated Convergence
Theorem, the continuous real-valued functions $\Psi_m(\bu,t)$ converge pointwise
in $(\bu,t)\in \Ccal_\loc(I,H_\rw)\times (I\setminus \{t_0\})$ to the
function $\Psi(\bu,t)$. Thus, $\Psi(\bu,t)$ is a Borel function
in $(\bu,t)\in \Ccal_\loc(I,H_\rw)\times (I\setminus \{t_0\})$. 

Now for each $\bu\in \Ccal_\loc(I,H_\rw)$, we note again that
\[ t \mapsto |\bu(t)|_{L^2}^2 - |\bu(t_0)|_{L^2}^2
\]
is uniformly bounded on compact intervals of $I$, 
so that
\[ \Psi(\bu,t) = \frac{1}{t-t_0}\int_{t_0}^t (|\bu(s)|_{L^2}^2 - |\bu(t_0)|_{L^2}^2)\;\rd s, 
     \qquad t\in I, \;t>t_0,
\]
is continuous in $t\in I\setminus \{t_0\}$.

Since $\bu\in \Ccal_\loc(I,H_\rw)$ is in particular weakly continuous at $t_0$, we have 
that 
\[ |\bu(t_0)|\leq \liminf_{t\rightarrow t_0^+} |\bu(t)|,
\]
which implies that
\[ \liminf_{t\rightarrow t_0^+} \Psi(\bu,t) \geq 0,
\]
proving \eqref{liminfpsigeqzero}.

The proof that $\bu\mapsto\liminf_{t\rightarrow t_0^+} \Psi(\bu,t)$ is Borel 
in $\bu\in \Ccal_\loc(I,H_\rw)$ uses the fact that $\Psi(\bu,t)$ is
continuous in $t\in I\setminus \{t_0\}$ for each $\bu$ fixed.
Note that
\[ \liminf_{t\rightarrow t_0^+} \Psi(\bu,t) 
      = \lim_{t\rightarrow 0} \inf_{t_0<s\leq t} \Psi(\bu,s)
      = \lim_{n\rightarrow \infty} \inf_{t_0<s\leq 1/n} \Psi(\bu,s),
\]
and since $\Psi(\bu,t)$ is continuous in $t\in I\setminus\{t_0\}$ we have
\[ \liminf_{t\rightarrow t_0^+} \Psi(\bu,t) 
      = \lim_{n\rightarrow 0} \inf_{t_0<s\leq t, \;s\in \QQ} \Psi(\bu,s).
\]
Since $(t_0,t]\cap \QQ$ is countable, the function
$\inf_{t_0<s\leq t, \;s\in \QQ} \Psi(\bu,s)$ is Borel, and hence
$\bu\mapsto  \liminf_{t\rightarrow t_0^+} \Psi(\bu,t)$ is the limit
of a sequence of Borel functions, hence it is also Borel. 

The equivalence between \eqref{liminfcritpsi}, \eqref{uinucalicrit}, and \eqref{limcritpsi} 
follows directly from Lemma \ref{criterionstrongcontinuityright}. 
\end{proof}
\medskip

\begin{prop}
  \label{ucalgdeltainucaltilde}
  Let $I\subset \RR$ be an interval closed and bounded on the left.
  Then $\Ucal_I(R)$ is a $\Gcal_\delta$-set in $\Ucal_I^\sharp(R)$,
  for any $R\geq R_0$.
\end{prop}

\begin{proof} 
Let $t_0$ be the left end point of the interval $I$.
We have from Lemma \ref{lemphicondition} that 
\[ \liminf_{t\rightarrow t_0^+} \Psi(\bu,t) \geq 0.
\]
for all $\bu\in \Ucal_I^\sharp(R)$, and that such $\bu$
belongs to $\Ucal_I(R)$ if and only if
\[ \liminf_{t\rightarrow t_0^+} \Psi(\bu,t) = 0.
\]
Hence, we can write the ``bad'' set $\Ucal_I^\sharp(R)\setminus \Ucal_I(R)$
as
\[ \Ucal_I^\rb(R) = \Ucal_I^\sharp(R)\setminus \Ucal_I(R)
    = \left\{\bu\in \Ucal_I^\sharp(R); \exists \;\delta,\varepsilon>0,
           \Psi(\bu,t)\geq \varepsilon, \forall t\in (t_0,t_0+\delta] \right\}.
\] 
We can also write it as a countable union of the form
\[ \Ucal_I^\rb(R) = \bigcup_{k,n\in \NN} \Ucal_I^{k,n}(R),
\]
where
\[ \Ucal_I^{k,n}(R) = \{\bu\in \Ucal_I^\sharp(R); \; 
     \Psi(\bu,t)\geq 1/k, \forall t\in (t_0,t_0+\tau/n] \},
\]
where $\tau>0$ is fixed and is such that $[t_0,t_0+\tau]\subset I$.
We claim that
\begin{equation}
  \label{badknsetisclosed}
  \Ucal_I^{k,n}(R) \emph{ is closed in } \Ucal_I^\sharp(R).
\end{equation}

In fact, since $\Ucal_I^\sharp(R)$ is metrizable, it suffices to work with 
sequences $\bu_j\in \Ucal_I^{k,n}(R)$ converging in this metric (i.e.
that of $\Ccal_\loc(I,B_H(R)_\rw)$) to some solution $\bu\in\Ucal_I^\sharp(R)$. 
Since $\Ucal_I^\sharp(R)$ is included in $L_\loc^2(I,H)$ (in fact, compactly
embedded, by Proposition \ref{ucalitildercompact}), it follows that this convergence
also takes place in $L^2(t_0,t_0+\tau;H)$. Since $\bu_j(t_0)$ converges weakly to
$\bu(t_0)$ in $H$, then 
\[ |\bu(t_0)|_{L^2}\leq \liminf_{j\rightarrow \infty} |\bu_j(t_0)|.
\]
This, together with the strong convergence in $L^2(t_0,t_0+\tau;H)$, implies that
\[ \Psi(\bu,t) \geq \limsup_{j\rightarrow\infty}\Psi(\bu_j,t) \geq \frac{1}{k}, 
     \qquad \forall t\in (t_0,t_0+\tau/n].
\]
Hence, $\bu\in \Ucal_I^{k,n}(R)$, which shows that this
set is closed in $\Ucal_I^\sharp(R)$. Note that since $\Ucal_I^\sharp(R)$ is 
compact in $\Ccal_\loc(I,B_H(R)_\rw)$ (from Proposition \ref{ucalitildercompact}), then so are
the sets $\Ucal_I^{k,n}(R)$.

Now, the complement $\Ucal_I^\sharp(R)\setminus \Ucal_I^{k,n}(R)$
is open in $\Ucal_I^\sharp(R)$ and we can write 
$\Ucal_I(R)=\Ucal_I^\sharp(R)\setminus\Ucal_I^\rb(R)$ 
as a countable intersection of open sets
\[ \Ucal_I(R) 
     = \bigcap_{k,n\in \NN} \left(\Ucal_I^\sharp(R)\setminus \Ucal_I^{k,n}(R)\right).
\]
Hence, $\Ucal_I(R)$ is a $\Gcal_\delta$-set in $\Ucal_I^\sharp(R)$.
\end{proof}
\medskip

Since $\Ucal_I = \bigcup_{k\in \NN}\Ucal_I(kR_0)$, in the case
$I$ is closed and bounded on the left, the following corollary holds.
\begin{cor}
  \label{corgdeltasigma}
  Let $I\subset\RR$ be an interval closed and bounded on the left.
  Then $\Ucal_I$ is a $\Gcal_{\delta\sigma}$-set in $\Ucal_I^\sharp$.
  \qed
\end{cor}

In the particular case the forcing term $\bbf$ is time-independent we also have the
following density result. We do not use this density in this article, but it might be useful in the study of stationary statistical solutions.

\begin{prop}
  Let $I\subset \RR$ be an interval closed and bounded on the left and let $R\geq R_0$.
  Suppose the forcing term is time-independent, with $\bbf\in H$. Then, 
  $\Ucal_I(R)$ is a dense $\Gcal_\delta$-set in $\Ucal_I^\sharp(R)$ and
  $\Ucal_I$ is a dense $\Gcal_{\delta\sigma}$-set in $\Ucal_I^\sharp$.
\end{prop}

\begin{proof}
  That $\Ucal_I(R)$ is a $\Gcal_\delta$-set in $\Ucal_I^\sharp(R)$ and
  $\Ucal_I$ is a $\Gcal_{\delta\sigma}$-set in $\Ucal_I^\sharp$ have been
  proved in Proposition \ref{ucalgdeltainucaltilde} and Corollary \ref{corgdeltasigma}.
  We only need to prove the density. Since any $\bu\in \Ucal_I^\sharp$ belongs to some
  $\Ucal_I^\sharp(R)$ for some $R\geq R_0$ sufficiently large, it suffices to prove the result
  for a given $R\geq R_0$. 

  Let then $\bu\in \Ucal_I^\sharp(R)$, with $R\geq R_0$. We first show the existence of 
  an extension of $\bu$ to a global solution $\tilde\bu\in \Ucal_{[t_0,\infty)}^\sharp(R)$ on 
  $[t_0,\infty)$. In the case $I$ is closed on the right, say $I=[t_0,T]$, just paste $\bu$ with 
  a global weak solution on $[T,\infty)$ with the initial condition $\bu(T)$,
  using Lemma \ref{pastinglemma}. In the case $I=[t_0,T)$,
  take a sequence $t_n$ of positive times converging to $T$ from the left,
  and, for each $n$, consider a global weak solutions on $[t_n,\infty)$ with
  initial condition $\bu(t_n)$. Then, from Lemma \ref{pastinglemma}, each such 
  global weak solution can be concatenated with $\bu$ on $[t_0,t_n]$ to yield a global 
  weak solution on $[t_0,\infty)$. Thanks now to Lemma \ref{convergenceofsolutions} 
  this sequence of concatenated solutions have a limit point which is a global weak 
  solution $\tilde\bu$ on $[t_0,\infty)$, which must agree with $\bu$ on $I$. Note
  that Lemma \ref{convergenceofsolutions} does not apply to $[t_0,\infty)$ since
  this interval is closed on the left, but since the sequence of solutions agree with $\bu$
  near the origin we just need to apply this lemma on, say, $(t_0,\infty)$.

  Let now $t_n'\rightarrow t_0^+$ be a sequence of positive times which are 
  points of strong continuity from the right for $\bu$. Define 
  $\tilde\bu_n(t)=\bu(t_n'-t_0+t)$, for $t\geq t_0$. Since $\bbf$ is time-independent,
  $\tilde\bu_n$ is also a weak solution on $[t_0,\infty)$ which belongs to
  $\Ucal_{[t_0,\infty)}(R)$. Let $\bu_n\in \Ucal_I(R)$ 
  be the restriction of $\tilde\bu_n$ to the time interval $I$.
  From the (uniform on bounded intervals) weak continuity of $\tilde\bu$
  it follows that $\bu_n$ converges to $\bu$ in $\Ucal_I^\sharp(R)$, 
  which shows that $\Ucal_I(R)$ is dense in $\Ucal_I^\sharp(R)$.
\end{proof}

\section{Time-dependent statistical solutions}

In this section we address the definition and properties of 
time-dependent statistical solutions and of Vishik-Fursikov
measures and Vishik-Fursikov statistical solutions of the Navier-Stokes equations, 
casting a new perspective on the theories given in
the fundamental works \cite{foias72,foiasprodi76,vishikfursikov78,vishikfursikov88}
(see also \cite{fmrt2001a}).

\subsection{Cylindrical test functions}
\label{cylindricaltestfn}

For the definition of statistical solutions one needs to consider
appropriate test functions in $\Ccal(H_\rw)$. For this purpose we use the
following definition.
\begin{defs}
  \label{deftestfunction}
  The cylindrical test functions are the functionals
  $\Phi:H\rightarrow \RR$ of the form
  \begin{equation}
    \label{cylindricalfunctions}
    \Phi(\bu) = \phi((\bu,\bv_1),\ldots,(\bu,\bv_k)),
  \end{equation}
  where $k\in\NN$, $\phi$ is a $C^1$ real-valued function on $\RR^k$
  with compact support, and $\bv_1,\ldots,\bv_k$ belong to $V$.
  For such $\Phi$, we denote by $\Phi'$ its Fr\'echet derivative
  in $H$, which has the form
  \[ \Phi'(\bu) = \sum_{j=1}^k
         \partial_j\phi((\bu,\bv_1),\ldots,(\bu,\bv_k))\bv_k,
  \]
  where $\partial_j\phi$ is the derivative of $\phi$ with respect
  to its $j$-th coordinate.
\end{defs}

The significance of the set of cylindrical test functions can be seen by the fact that, when 
restriced to a bounded ball $B_H(R)_\rw$, $R>0$, the cylindrical test functions are 
dense in the space $\Ccal(B_H(R)_\rw)$. In fact, since $B_H(R)_\rw$ is a
compact and separable metrizable space, and the cylindrical test functions form a 
subalgebra which contains the unit element and separates points in $B_H(R)_\rw$,
it follows from the Stone-Weierstrass Theorem \cite{dunfordschwartz} 
that the space of cylindrical test 
functions restricted to $B_H(R)_\rw$ is dense in $\Ccal(B_H(R)_\rw)$.

\subsection{Definition and existence of time-dependent statistical solutions}
Time-dependent statistical solutions in the sense of \cite{foias72, foiasprodi76} 
are defined in the following way.

\renewcommand{\theenumi}{\roman{enumi}}
\begin{defs}
  \label{deftimedependetstatisticalsolution}
  For a given interval $I\subset \RR$, 
  a family $\{\mu_t\}_{t\in I}$ 
  of Borel probability measures on $H$ is called a statistical
  solution of the Navier-Stokes equations over $I$
  if the following conditions hold:
  \begin{enumerate}
    \item \label{ssphimeas} The function
      \[ t\mapsto \int_H \Phi(\bu) \;\rd\mu_t(\bu)
      \]
      is measurable on $I$ for every bounded and continuous
      real-valued function $\Phi$ on $H$;
    \item \label{ssenergy} The function
      \[  t\mapsto \int_H |\bu|_{L^2}^2 \;\rd\mu_t(\bu)
      \]
      belongs to $L_\loc^\infty(I)$;
    \item \label{ssenstrophy} The function
      \[ t\mapsto \int_H \|\bu\|_{H^1}^2 \;\rd\mu_t(\bu)
      \]
      belongs to $L_\loc^1(I)$;
    \item \label{ssliouville} For any cylindrical test function $\Phi$, 
      the Liouville-type equation
      \begin{equation}
        \label{liouvilleeq}
        \int_H \Phi(\bu) \;\rd\mu_t(\bu) = \int_H \Phi(\bu) \;\rd\mu_{t'}(\bu)
           + \int_{t'}^t \int_H \dual{\bF(\bu),\Phi'(\bu)}_{V',V}\;\rd\mu_s(\bu)\;\rd s
      \end{equation}
      holds for all $t', t\in I$, where
      $\bF(\bu) = \bbf - \nu A\bu - B(\bu,\bu)$, so that 
      \[ \dual{\bF(\bu),\Phi'(\bu)}_{V',V}=(\bbf,\Phi'(\bu))_{L^2}-\nu
           \Vinner{\bu,\Phi'(\bu)}-b(\bu,\bu,\Phi'(\bu));
      \]
    \item \label{ssstmeaneneineq}
      The strengthened mean energy inequality holds on $I$, i.e.
      there exists a set $I'\subset I$ of full measure in $I$ such that
      for any nonnegative, nondecreasing, continuously-differentiable real-valued function 
      $\psi:[0,\infty)\rightarrow \RR$ with bounded derivative, the inequality
      \begin{multline}
        \label{strengthenedmeanenergyineq}
        \frac{1}{2}\int_H \psi(|\bu|_{L^2}^2)\;\rd\mu_t(\bu)
         + \nu\int_{t'}^t \int_H \psi'(|\bu|_{L^2}^2)\|\bu\|_{H^1}^2
                 \;\rd\mu_s(\bu)\rd s \\
         \leq \frac{1}{2}\int_H \psi(|\bu|_{L^2}^2)\;\rd\mu_{t'}(\bu)
         + \int_{t'}^t \int_H \psi'(|\bu|_{L^2}^2)(\bbf,\bu)_{L^2}
                 \;\rd\mu_s(\bu)\rd s
      \end{multline}
      holds for any $t'\in I'$ and for all $t\in I$ with $t\geq t'$;
    \item \label{ssinitialtime}
      If $I$ is closed and bounded on the left with left end point $t_0$,
      then the function
      \[ t\mapsto \int_H \psi(|\bu|_{L^2}^2) \;\rd\mu_t(\bu)
      \]
      is continuous at $t=t_0$ from the right, for every
      function $\psi$ as in \eqref{ssstmeaneneineq}.
 \end{enumerate}
\end{defs}

\renewcommand{\theenumi}{\alph{enumi}}
\begin{rmk}
  \label{r2.6a}
  We have the following consequences and remarks concerning the definition above.
  \begin{enumerate}
    \item Note that \eqref{ssphimeas}, \eqref{ssenergy}, \eqref{ssenstrophy}, 
      \eqref{ssliouville} imply that the function 
      $t\rightarrow \int_H \Phi(\bu)\;\rd\mu_t(\bu)$ is continuous
      on $I$ for all $\Phi$ of the type considered in \eqref{ssliouville}. 
      In particular, this continuity is valid at the initial 
      time if the initial time belongs to the interval $I$, i.e.
      if $I$ is closed and bounded on the left.

    \item \label{comb} 
      Note that condition \eqref{ssenstrophy} actually follows from conditions 
      \eqref{ssphimeas}, \eqref{ssenergy}, and \eqref{ssstmeaneneineq}
      with $\psi(r)=r$, $r\geq 0$.

    \item \label{comc} 
      If $I$ is closed and bounded on the left, and assuming 
      \eqref{ssphimeas}, \eqref{ssenergy}, \eqref{ssenstrophy}, 
      \eqref{ssliouville}, \eqref{ssstmeaneneineq}
      then condition \eqref{ssinitialtime} above is equivalent to assuming that
      \eqref{ssstmeaneneineq} holds for $t'=t_0$. 
      Indeed, it is immediate to see, using
      \eqref{ssphimeas}, \eqref{ssenergy}, \eqref{ssenstrophy} 
      and letting $t'\rightarrow t_0$ in \eqref{ssstmeaneneineq}, that \eqref{ssinitialtime}
      implies \eqref{ssstmeaneneineq} with $t'=t_0$. 

      For the converse, let $P_m$ be the Galerkin projector on the 
      first $m$ modes of the Stokes operator and let $\psi_\varepsilon$
      be a continuously-differentiable approximation of $\psi$ with
      compact support and with $0\leq \psi_\varepsilon\leq \psi$.
      Then, $\Phi_{m,\varepsilon}(\bu) =\psi_\varepsilon(|P_m\bu|_{L^2}^2)$ 
      satisfies the requirements in \eqref{ssliouville}, and
      \[ \lim_{t\rightarrow t_0^+} 
            \int_H \psi_\varepsilon(|P_m\bu|_{L^2}^2) \;\rd\mu_t(\bu) 
              = \int_H \psi_\varepsilon(|P_m\bu|_{L^2}^2) \;\rd\mu_{t_0}(\bu).
      \]
      Using this relation and the facts that $\psi_\varepsilon\leq \psi$
      and $|P_m\bu|_{L^2}\leq |\bu|_{L^2}$ in $H$ (so that in particular we can
      apply the Monotone Convergence Theorem 
      \cite{dunfordschwartz,rudin}) we find
      \begin{multline*}
          \int_H \psi(|\bu|_{L^2}^2) \;\rd\mu_{t_0}(\bu) 
            = \lim_{\varepsilon\rightarrow 0, m\rightarrow \infty} 
                \int_H \psi_\varepsilon(|P_m\bu|_{L^2}^2)
                \;\rd\mu_{t_0}(\bu) \\
            = \lim_{\varepsilon\rightarrow 0, m\rightarrow \infty} 
               \lim_{t\rightarrow t_0^+}
               \int_H \psi_\varepsilon(|P_m\bu|_{L^2}^2)
                \;\rd\mu_t(\bu) 
            \leq \liminf_{t\rightarrow t_0^+} 
               \int_H \psi(|\bu|_{L^2}^2) \;\rd\mu_t(\bu).
      \end{multline*}
    
      On the other hand, using \eqref{ssstmeaneneineq} with $t'=t_0$ and letting 
      $t\rightarrow t_0^+$, we find
      \[  \limsup_{t\rightarrow t_0^+} 
            \int_H \psi(|\bu|_{L^2}^2) \;\rd\mu_t(\bu) 
              \leq \int_H \psi(|\bu|_{L^2}^2) \;\rd\mu_{t_0}(\bu).
      \]    

      Therefore, the equality holds and \eqref{ssinitialtime} is proved. 
 
    \item Taking $\psi(r)=r$ in condition \eqref{ssinitialtime} we find that
      \begin{equation}
        \label{e2.N}
        \lim_{t\rightarrow t_0^+}\int_H |\bu|_{L^2}^2 \;\rd\mu_t(\bu) 
            = \int_H |\bu|_{L^2}^2 \;\rd\mu_{t_0}(\bu).
      \end{equation}
      Then, using \eqref{ssenergy}, we find that the initial condition 
      necessarily has finite mean kinetic energy:
      \begin{equation}
        \int_H |\bu|_{L^2}^2 \;\rd\mu_{t_0}(\bu) < \infty.
        \label{e2.Na}
      \end{equation}

    \item Using the arguments in \eqref{comb} and \eqref{comc}, one can show
      in fact that $\int \psi(\bu)\;\rd\mu_t(\bu)$ and
      $\int |\bu|_{L^2}^2 \;\rd\mu_t(\bu)$ are continuous from the
      right almost everywhere in $I$, i.e. at any point $t'$ 
      allowed in \eqref{ssstmeaneneineq}. 
      
    \item A weaker form of the energy inequality (replacing \eqref{ssstmeaneneineq}) 
      is sometimes used in the definition of statistical solution, namely
      \begin{multline}
        \label{meanenergyinequality}
        \frac{1}{2}\int_H |\bu|_{L^2}^2\;\rd\mu_t(\bu)
         + \nu\int_{t'}^t \int_H \|\bu\|_{H^1}^2 \;\rd\mu_s(\bu)\rd s \\
         \leq \frac{1}{2}\int_H |\bu|_{L^2}^2\;\rd\mu_{t'}(\bu)
         + \int_{t'}^t \int_H (\bbf,\bu)_{L^2} \;\rd\mu_s(\bu)\rd s
      \end{multline}
      for all $t\in I$ and almost all $t'\in I$ with $t'\leq t$. For individual solutions,
      the two corresponding energy inequalities, \eqref{energyinequalityintegralform} 
      and \eqref{strengthenedenergyineq}, are actually equivalent, as proved 
      in \cite{dascaliuc}. It is not known, however, whether the 
      mean versions \eqref{meanenergyinequality} and \eqref{strengthenedmeanenergyineq} 
      are equivalent. The strengthened version \eqref{ssstmeaneneineq} 
      is consistent with the usual definition 
      of stationary statistical solutions (see \cite{frtssp2}).
  \end{enumerate}
\end{rmk}

The existence of time-dependent statistical solutions in the sense above
was first proved in \cite{foias72} via Galerkin approximation (see Theorem 1 on page 254
and Proposition 1 on page 291 in \cite{foias72}). The 
existence result can be stated in the following way: Let $t_0\in \RR$ and let $\mu_0$ be a 
Borel Probability measure on $H$ satisfying 
  \[  \int_H |\bu|_{L^2}^2\;\rd\mu_0(\bu) < \infty.
  \]
Then, there exists a time-dependent statistical solution  $\{\mu_t\}_{t\geq t_0}$ satisfying 
$\mu_{t_0}=\mu_0$.

We present in the next sections a different proof based on an idea given in 
\cite{fmrt2001a}, which yields in fact a statistical solution of a particular type, which
we term a Vishik-Fursikov statistical solution. 

\subsection{Definition and existence of Vishik-Fursikov measures}

In the approach of Vishik and Fursikov, the statistical solutions
are obtained through the help of probability measures in suitable 
trajectory spaces. What makes them measures relevant to 
fluid flows is the condition that they be carried by the space
of individual weak solutions. We also ask them to have
finite mean kinetic energy. Inspired by their approach we introduce the
following definition.

\renewcommand{\theenumi}{\roman{enumi}}
\begin{defs}
  \label{defvfmeasure}
  For a given interval $I\subset \RR$, a
  Vishik-Fursikov measure over $I$  
  is defined as a Borel probability measure $\rho$ on the space
  $\Ccal_\loc(I,H_\rw)$ with the following properties
  \begin{enumerate}
    \item \label{rhocarrierforvfm} $\rho$ is carried by $\Ucal_I^\sharp$;
    \item \label{rhomeankineticforvfm} We have
      \[ t\rightarrow \int_{\Ucal_I^\sharp} |\bu(t)|_{L^2}^2 \;\rd\rho(\bu) \in L_\loc^\infty(I);
      \]
    \item \label{meancontinuityatinitialtimeforvfm}
      If $I$ is closed and bounded on the left, with
      left end point $t_0$, then for any nonnegative, nondecreasing 
      continuously-differentiable real-valued function $\psi:[0,\infty)\rightarrow \RR$ 
      with bounded derivative, we have
      \[ \lim_{t\rightarrow t_0^+} \int_{\Ucal_I^\sharp} \psi(|\bu(t)|_{L^2}^2) \;\rd\rho(\bu)
               = \int_{\Ucal_I^\sharp} \psi(|\bu(t_0)|_{L^2}^2) \;\rd\rho(\bu) < \infty.
      \]
  \end{enumerate}
\end{defs}

\renewcommand{\theenumi}{\roman{enumi}}
\begin{rmk}
  \label{rmkorgdefvfss}
  The Definition \ref{defvfmeasure} for a Vishik-Fursikov measure is 
  inspired by the definition of time-dependent statistical solution given
  by Vishik and Fursikov (see \cite{vishikfursikov88}), but it is actually not 
  the same. Our definition is stricter in the sense that any Vishik-Fursikov 
  measure according to Definition \ref{defvfmeasure} is a time-dependent 
  statistical solution according to Vishik and Fursikov. 
  Considering for simplicity $I=[0,T]$, $T>0$,
  their original definition of time-dependent statistical solution 
  on $[0,T]$ is that of a Borel probability measure $\rho$ on the space
  $Z=L^2(0,T;H)\cap C([0,T]; V^{-s})$, for a given $s\geq 2$, 
  where $V^\alpha = D(A^{\alpha/2})$, $\alpha\in \RR$, 
  and satisfying the following conditions:
  \begin{enumerate}
    \item $\rho$ is carried by the space
      \[ \tilde \Lcal  = \{ \bu\in L^2(0,T; V)\cap L^\infty(0,T; H); 
            \; \bu'\in L^\infty(0,T;V^{-s});
      \]
    \item $\rho$ is carried by a Borel subset $W$ of $Z$ which is 
      closed in $\tilde \Lcal$, and consists of weak solutions not necessarily 
      of Leray-Hopf type;  
    \item $\rho$ satisfies a weaker form of the mean energy inequality, namely     
      \[ \int_Z \left( |\bu(t)|_{L^2}^2 + \|\bu\|_{\tilde \Lcal}^2 \right\} \;\rd\rho(\bu)
           \leq C \left( \int_Z |\bu(0)|_{L^2}^2 \;\rd\rho(\bu) + 1\right),
      \]
      for all $t\in [0,T]$, for some constant $C\geq 0$ which does not depend on $t$,
      and where $\|\bu\|_{\tilde \Lcal}$ is the natural norm for $\tilde \Lcal$.
  \end{enumerate}
  It has been proved by Vishik and Fursikov (see \cite[Chapter 4]{vishikfursikov88}),
  using Galerkin approximations, 
  that for any $T>0$ and for every Borel probability $\mu_0$ on $H$ with finite mean 
  kinetic energy, i.e.
  \[ \int_H |\bu|_{L^2}^2\;\rd\mu_0(\bu) < \infty,
  \] 
  there exists a time-dependent statistical solution $\rho$ over 
  the time interval $[0,T]$ with initial probability $\Pi_0\rho=\mu_0$.
  This notion of statistical solution will not be used in the sequel.
\end{rmk}

\begin{rmk}
  \label{rmkexistenceproofsfortdss}
  As mentioned in the previous section, the original proof of existence of statistical solutions 
  in the sense of Definition \ref{deftimedependetstatisticalsolution} was also based on
  Galerkin approximations \cite{foias72}. 
  In \cite{fmrt2001a}, a different proof was given, which form the basis of our present work. 
  The proof in \cite{fmrt2001a} was based on the 
  Krein-Milman Theorem, in which the idea was essentially to prove the existence of 
  a Vishik-Fursikov measure in a sense somewhat weaker than Definition \ref{defvfmeasure}
  and then prove that the projections of such a measure at each time $t$
  yield a family of measures which is a statistical solution.
  In that proof, however, we overlooked the fact that $\Ucal_I(R)$ may not be compact, and
  in particular the energy inequality for $t_0$ at the left end point of the time interval
  does not follow in a trivial way. 
  Using instead the compact space $\Ucal^\sharp_I(R)$, the proof of existence of a 
  Vishik-Fursikov measure follows very much the steps in \cite{fmrt2001a}; nevertheless, 
  we include below this proof for the sake of completeness. The only delicate step and 
  which was not proved in \cite{fmrt2001a} is the continuity of the moments of the measure 
  at the initial time (expressed by condition \eqref{meancontinuityatinitialtimeforvfm} 
  in the Definition \ref{defvfmeasure}),
  and for that we need to take special care. 
  We also remark on another subtle difference between the result in \cite{fmrt2001a}, 
  in which the specific construction yielded a family of measures satisfying a mean 
  energy inequality, while in our case we prove that any Vishik-Fursikov measure satisfies 
  a (strengthened) mean energy inequality, not only the one constructed in the proof of 
  existence.
\end{rmk}

\begin{thm}
  \label{thmexistencevfmeasure}
  Let $t_0\in \RR$ and let $\mu_0$ be a Borel probability 
  measure on $H$ with finite mean kinetic energy, i.e.
  \[ \int_H |\bu|_{L^2}^2 \;\rd\mu_0(\bu) < \infty.
  \]
  Then, there exists a Vishik-Fursikov measure $\rho$ over the time interval
  $I=[t_0,\infty)$ such that $\Pi_{t_0}\rho=\mu_0$.
\end{thm}

\begin{proof} 
Let us first consider the case in which $\mu_0$ is carried
by a ball $B_H(R)$ with $R\geq R_0$, i.e. $\mu_0(B_H(R))=1$.

Consider the measures $P_m\mu_0$ given by $P_m\mu_0(E)=\mu_0(P_m^{-1} E)$,
for all Borel sets $E$ in $H$, where $P_m$, $m\in \NN$, are the Galerkin projectors. 
By the Lebesgue Dominated Convergence Theorem, we have 
\begin{equation}
  \label{mu0mconvergestomu0}
  \int_H \varphi(\bu) \;\rd P_m\mu_0(\bu) = \int_H \varphi(P_m\bu) \;\rd\mu_0(\bu) 
    \rightarrow \int_H \varphi(\bu)\;\rd\mu_0(\bu),
      \quad \forall \varphi\in \Ccal(B_H(R)_\rw), 
\end{equation}
which means that
\[ P_m\mu_0 \stackrel{*}{\rightharpoonup} \mu_0, \quad \text{as } m\rightarrow \infty.
\]

Since $P_mB_H(R)_\rw$ is compact and separable 
and $P_m\mu_0$ is carried by $P_m B_H(R)_\rw$,
it follows by the Krein-Milman Theorem that, for each $m\in \NN$, there exist
$J^{m,n}\in \NN$, $\theta_j^{m,n}\in \RR$, and $\bu_{0,j}^{m,n}\in P_m B_H(R)$
such that
\[ 0 < \theta_j^{m,n} \leq 1, \quad \sum_{j=1}^{J^{m,n}} \theta_j^{m,n} = 1,
\]
and
\begin{equation}
  \label{mu0mnconvergestomu0m}
  \mu_0^{m,n} \define \sum_{j=1}^{J^{m,n}} \theta_j^{m,n} \delta_{\bu_{0,j}^{m,n}} 
     \stackrel{*}{\rightharpoonup} P_m\mu_0, \quad \text{as } m\rightarrow \infty.
\end{equation}
Since for any nonnegative, nondecreasing, continuously-differentiable real-valued function 
$\psi:[0,\infty)\rightarrow \RR$ with bounded derivative, the function defined by 
\[ \varphi(\bu)=\psi(|P_m\bu|_{L^2}^2)
\]
belongs to $\Ccal(B_H(R)_\rw)$ for each $m\in \NN$, 
we have in particular that
\[  \lim_{n\rightarrow\infty}\int_H \psi(|P_m\bu|_{L^2}^2)\;\rd\mu_0^{m,n}(\bu) 
      = \int_H \psi(|P_m\bu|_{L^2}^2)\;\rd P_m\mu_0(\bu)
      = \int_H \psi(|P_m\bu|_{L^2}^2)\;\rd\mu_0(\bu).
\]
Since $\mu_0^{m,n}$ is carried by $P_m B_H(R)$, we have 
$|P_m\bu|_{L^2}=|\bu|_{L^2}$ for $\mu_0^{m,n}$-almost every $\bu$, so that
\begin{multline}
  \label{mu0mnpsiconvergestomu0m}
  \lim_{n\rightarrow \infty} \int_H \psi(|\bu|_{L^2}^2)\;\rd\mu_0^{m,n}(\bu) 
    = \lim_{n\rightarrow \infty} \int_H \psi(|P_m\bu|_{L^2}^2)\;\rd\mu_0^{m,n}(\bu) \\
    = \int_H \psi(|P_m\bu|_{L^2}^2)\;\rd\mu_0(\bu)
    \leq \int_H \psi(|\bu|_{L^2}^2)\;\rd\mu_0(\bu), \qquad \forall m\in \NN.
\end{multline}

Since $B_H(R)_\rw$ is compact and metrizable, the space $\Ccal(B_H(R)_\rw)$
is separable, and there exists a countable dense set $\{\varphi_\ell\}_{\ell\in \NN}$ 
in $\Ccal(B_H(R)_\rw)$. Let also $\{\psi_\ell\}_\ell$ be a countable dense set,
with respect to the uniform topology of $\Ccal([0,\infty),\RR)$, in the space
of nonnegative, nondecreasing, continuously-differentiable real-valued functions 
with bounded derivative. 

For each $p\in \NN$, choose, thanks to \eqref{mu0mconvergestomu0}, an index
$m_p\in \NN$ such that 
\begin{equation}
  \label{mu0mpellconvergesmu0ell}
  \left| \int_H \varphi_\ell(\bu) \;\rd P_{m_p}\mu_0(\bu) - \int_H \varphi_\ell(\bu)\;\rd\mu_0(\bu) \right|
    \leq \frac{1}{2p}, \quad \forall \ell=1,\ldots, p.
\end{equation}
Then, using \eqref{mu0mnconvergestomu0m} and \eqref{mu0mnpsiconvergestomu0m}, 
choose $n_p$ such that
\begin{equation}
  \label{mu0mnpellconvergesmu0mpell}
  \left| \int_H \varphi_\ell(\bu) \;\rd\mu_0^{m_p,n_p}(\bu) 
    - \int_H \varphi_\ell(\bu)\;\rd P_{m_p}\mu_0(\bu) \right|
    \leq \frac{1}{2p}, \quad \forall \ell=1,\ldots, p,
\end{equation}
and 
\begin{equation}
  \int_H \psi_\ell(|\bu|_{L^2}^2)\;\rd\mu_0^{m_p,n_p}(\bu) 
    \leq \int_H \psi_\ell(|\bu|_{L^2}^2)\;\rd\mu_0(\bu) + \frac{1}{p}, \quad \forall \ell=1,\ldots, p.
\end{equation}

Set $\mu_0^p=\mu_0^{m_p,n_p}$. Using \eqref{mu0mpellconvergesmu0ell} 
and \eqref{mu0mnpellconvergesmu0mpell}, it follows that 
for each $\ell\in \NN$ and $p\geq \ell$,
\begin{equation}
  \label{mu0pelllimsupmu0}
  \left| \int_H \varphi_\ell(\bu) \;\rd \mu_0^p(\bu) - \int_H \varphi_\ell(\bu)\;\rd\mu_0(\bu) \right|
    \leq \frac{1}{2p}+\frac{1}{2p} = \frac{1}{p}.
\end{equation}
Since $\{\varphi_\ell\}_\ell$ is dense in $\Ccal(B_H(R)_\rw)$ and the measures are
probabilities, we find that
\begin{equation}
  \int_H \varphi(\bu)\;\rd\mu_0^p(\bu) \rightarrow \int_H \varphi(\bu)\;\rd\mu_0(\bu), 
    \quad \text{as } p\rightarrow \infty, \;\forall \varphi\in \Ccal(B_H(R)_\rw),
\end{equation}
which means that
\begin{equation}
  \label{mu0pconvergestomu0}
  \mu_0^p  \stackrel{*}{\rightharpoonup} \mu_0, \quad \text{as } p\rightarrow \infty.
\end{equation}
Moreover, it follows from \eqref{mu0pelllimsupmu0} and the density of $\{\psi_\ell\}_\ell$
that
\begin{equation}
  \label{mu0ppsilimsuptomu0}
  \limsup_{p\rightarrow\infty} \int_H \psi(|\bu|_{L^2}^2) \;\rd\mu_0^p(\bu)
     \leq \int_H \psi(|\bu|_{L^2}^2) \;\rd\mu_0(\bu),
\end{equation}
for all nonnegative, nondecreasing, continuously-differentiable real-valued functions 
$\psi:[0,\infty)\rightarrow \RR$ with bounded derivative. We now use properties
\eqref{mu0pconvergestomu0}, and \eqref{mu0ppsilimsuptomu0}, along with
the fact that $\mu_0^p$ is a convex combination of Dirac deltas
(given by \eqref{mu0mnconvergestomu0m}), to construct 
a Vishik-Fursikov measure.

For each $p\in \NN$, let $\bu^p$ be a weak solution on $[t_0,\infty)$ 
with initial condition $\bu_j^p(t_0)=\bu_{0,j}^{m_p,n_p}$. Since $\bu_{0,j}^{m_p,n_p}$
is bounded in $H$ by $R$, with $R\geq R_0$, it follows from the energy estimate 
\eqref{energyestimate} that $\bu^p(t)\in B_H(R)$ for all $t\geq t_0$, $p\in \NN$, 
$j=1,\ldots, J^{m_p,n_p}$.

Consider now the measures
\[ \rho^p = \sum_{j=1}^{J^{m_p,n_p}} \theta_j^{m_p,n_p} \delta_{\bu_j^p},
\]
which are Borel probability measures on $\Ccal_\loc([t_0,\infty),H_\rw)$ carried by
$\Ucal_{[t_0,\infty)}^\sharp(R)$. According to Proposition \ref{ucalitildercompact}, 
the space $\Ucal_{[t_0,\infty)}^\sharp(R)$ is a compact metric space. 
Therefore, there exists $\rho\in \Ucal_{[t_0,\infty)}^\sharp(R)$ which is
the weak-star limit of the sequence (or a subsequence if necessary) 
of measures $\{\rho^p\}_p$, i.e.
\begin{equation}
  \label{rhopconvergestorho}
  \rho^p \stackrel{*}{\rightharpoonup} \rho, \quad \text{as } p\rightarrow \infty.
\end{equation}

The measure $\rho$ is our candidate for the desired Vishik-Fursikov measure.
We already have that $\rho$ is carried by $\Ucal_{[t_0,\infty)}^\sharp(R)
\subset \Ucal_{[t_0,\infty)}^\sharp$, so that condition \eqref{rhocarrierforvfm} of the
Definition \ref{defvfmeasure} of a Vishik-Fursikov measure
is satisfied. Since $(\bu,t)\mapsto |\bu(t)|_{L^2}^2$ is a Borel function on
$\Ccal_\loc([t_0,\infty),H_\rw)\times [t_0,\infty)$ and
\[ \int_{\Ucal_{[t_0,\infty)}^\sharp} |\bu(t)|_{L^2}^2\;\rd\rho(\bu)
     = \int_{\Ucal_{[t_0,\infty)}^\sharp(R)} |\bu(t)|_{L^2}^2\;\rd\rho(\bu)
     \leq R,
\]
we see that condition \eqref{rhomeankineticforvfm} 
of the Definition \ref{defvfmeasure} is also satisfied. 

Before proving condition \eqref{meancontinuityatinitialtimeforvfm}, 
let us show that the initial condition
$\Pi_{t_0}\rho=\mu_0$ is satisfied. Indeed, we have from 
\eqref{mu0pconvergestomu0} that 
$\Pi_{t_0}\rho^p = \mu^p \stackrel{*}{\rightharpoonup} \mu_0$
on $B_H(R)_\rw$.
On the other hand, since $\Pi_{t_0}$ is weakly continuous,
we have from \eqref{rhopconvergestorho} that 
$\Pi_{t_0}\rho^p\stackrel{*}{\rightharpoonup}\Pi_{t_0}\rho$.
Thus, 
\[ \Pi_{t_0}\rho=\mu_0.
\]

We now prove condition \eqref{meancontinuityatinitialtimeforvfm}. 
Let $\psi:[0,\infty)\rightarrow \RR$ be a nonnegative, 
nondecreasing, continuously-differentiable real-valued function with bounded derivative.
Since any $\bu\in \Ucal_{[t_0,\infty)}^\sharp(R)$ is weakly continuous at $t_0$, it
follows from the Fatou Lemma that
\begin{equation}
  \label{lscpsirho}
   \int_{\Ucal_{[t_0,\infty)}^\sharp(R)} \psi(|\bu(t_0)|_{L^2}^2) \;\rd\rho(\bu) 
     \leq \liminf_{t\rightarrow t_0^+} \int_{\Ucal_{[t_0,\infty)}^\sharp(R)} \psi(|\bu(t)|_{L^2}^2) 
       \;\rd\rho(\bu).
\end{equation}

Now, since $\varphi(\bu)=\psi(|P_m\bu(t)|_{L^2}^2)$ is continuous on $\Ucal_{[t_0,\infty)}^\sharp(R)$ 
for each $m\in \NN$, we have that
\begin{multline*}
  \int_{\Ucal_{[t_0,\infty)}^\sharp(R)} \psi(|P_m\bu(t)|_{L^2}^2)\;\rd\rho(\bu)
      = \lim_{p\rightarrow\infty}\int_{\Ucal_{[t_0,\infty)}^\sharp(R)} 
              \psi(|P_m\bu(t)|_{L^2}^2)\;\rd\rho^p(\bu) \\
      \leq \liminf_{p\rightarrow\infty}\int_{\Ucal_{[t_0,\infty)}^\sharp(R)} 
              \psi(|\bu(t)|_{L^2}^2)\;\rd\rho^p(\bu).
\end{multline*}
Using then the Monotone Convergence Theorem we find that
\begin{multline*}
  \int_{\Ucal_{[t_0,\infty)}^\sharp(R)} \psi(|\bu(t)|_{L^2}^2)\;\rd\rho(\bu)
     = \lim_{m\rightarrow\infty} \int_{\Ucal_{[t_0,\infty)}^\sharp(R)} \psi(|P_m\bu(t)|_{L^2}^2)\;\rd\rho(\bu) \\
     \leq \liminf_{p\rightarrow\infty}\int_{\Ucal_{[t_0,\infty)}^\sharp(R)} \psi(|\bu(t)|_{L^2}^2)\;\rd\rho^p(\bu).
\end{multline*}      

Using the fact that $\rho^p$ is actually carried by $\Ucal_{[t_0,\infty)}(R)$
(since it is a convex combination of Dirac deltas carried by weak solutions 
defined on $[t_0,\infty)$), the strengthened energy inequality \eqref{strengthenedenergyineq}
holds $\rho^p$-almost everywhere and we find the uniform estimate
\[ \psi(|\bu(t)|_{L^2}^2) \leq \psi(|\bu(t_0)|_{L^2}^2) 
  + \frac{1}{\nu\lambda_1}\|\bbf\|_{L^\infty(t_0,\infty;H)}^2(\sup_{r\in \RR}\psi'(r))(t-t_0),
\]
for $\rho^p$-almost every $\bu$.
Thus, 
\begin{multline}
  \label{uscpsirhop}
  \int_{\Ucal_{[t_0,\infty)}^\sharp(R)} \psi(|\bu(t)|_{L^2}^2)\;\rd\rho^p(\bu)
     \leq \int_{\Ucal_{[t_0,\infty)}^\sharp(R)} \psi(|\bu(t_0)|_{L^2}^2)\;\rd\rho^p(\bu) \\
       + \frac{1}{\nu\lambda_1}\|\bbf\|_{L^\infty(t_0,\infty;H)}^2(\sup_{r\in \RR}\psi'(r))(t-t_0).
\end{multline}
Therefore,
\begin{equation}
  \label{uscpsirho}
  \begin{aligned}
  \limsup_{t\rightarrow t_0^+} & \int_{\Ucal_{[t_0,\infty)}^\sharp(R)} \psi(|\bu(t)|_{L^2}^2)
    \;\rd\rho(\bu) \\
     & \leq \limsup_{t\rightarrow t_0^+}
        \liminf_{p\rightarrow\infty}\int_{\Ucal_{[t_0,\infty)}^\sharp(R)} \psi(|\bu(t)|_{L^2}^2)
          \;\rd\rho^p(\bu) \\
      & \leq \limsup_{t\rightarrow t_0^+} \left( 
         \liminf_{p\rightarrow\infty}\int_{\Ucal_{[t_0,\infty)}^\sharp(R)} \psi(|\bu(t_0)|_{L^2}^2)
           \;\rd\rho^p(\bu) \right. \\
      & \qquad \left.            + \frac{1}{\nu\lambda_1}\|\bbf\|_{L^\infty(t_0,\infty;H)}^2
              \sup_{r\in \RR}\psi'(r)(t-t_0)\right) \\
      & = \liminf_{p\rightarrow\infty}\int_{\Ucal_{[t_0,\infty)}^\sharp(R)} \psi(|\bu(t_0)|_{L^2}^2)
        \;\rd\rho^p(\bu) \\
      & \leq \int_{\Ucal_{[t_0,\infty)}^\sharp(R)} \psi(|\bu(t_0)|_{L^2}^2)\;\rd\rho(\bu).
  \end{aligned}
\end{equation}
Putting \eqref{lscpsirho} and \eqref{uscpsirho} together proves condition 
\eqref{meancontinuityatinitialtimeforvfm}. 
This completes the proof in the case in which $\mu_0$ is carried
by a bounded set in $H$. Note also that taking $\psi$ identically $1$ in \eqref{uscpsirhop}
and considering the Galerkin projetor $P_m$ we have
\[  \int_{\Ucal_{[t_0,\infty)}^\sharp(R)} |P_m\bu(t)|_{L^2}^2\;\rd\rho^p(\bu)
     \leq \int_{\Ucal_{[t_0,\infty)}^\sharp(R)} |\bu(t_0)|_{L^2}^2\;\rd\rho^p(\bu) \\
       + \frac{1}{\nu\lambda_1}\|\bbf\|_{L^\infty(t_0,\infty;H)}^2(t-t_0).
\]
Then, passing to the limit as $p\rightarrow \infty$ using \eqref{mu0ppsilimsuptomu0} and
the fact that $\bu\mapsto |P_m\bu(t)|_{L^2}^2$ is continuous, we find
\[ \int_{\Ucal_{[t_0,\infty)}^\sharp(R)} |P_m\bu(t)|_{L^2}^2\;\rd\rho(\bu)
     \leq \int_{\Ucal_{[t_0,\infty)}^\sharp(R)} |\bu(t_0)|_{L^2}^2\;\rd\rho(\bu) \\
       + \frac{1}{\nu\lambda_1}\|\bbf\|_{L^\infty(t_0,\infty;H)}^2(t-t_0).
\]
Taking the limit as $m\rightarrow\infty$ and using the Monotone Convergence
Theorem we obtain
\begin{multline}
  \label{energyestimateforrhoR}
  \int_{\Ucal_{[t_0,\infty)}^\sharp(R)} |\bu(t)|_{L^2}^2\;\rd\rho(\bu)
     \leq \int_{\Ucal_{[t_0,\infty)}^\sharp(R)} |\bu(t_0)|_{L^2}^2\;\rd\rho(\bu) \\
     + \frac{1}{\nu\lambda_1}\|\bbf\|_{L^\infty(t_0,\infty;H)}^2(t-t_0), \;\forall t\geq t_0.
\end{multline}

Consider now the case in which $\mu_0$ is not carried by any bounded set in $H$.
Then, there exists an increasing sequence $\{R_k\}_k$, $R_k\geq R_0$, with 
$R_k\rightarrow \infty$ and such that $\mu_0(A_k)>0$ for every $k$, where
$A_1=B_H(R_1)$ and $A_k=B_H(R_k)\setminus B_H(R_{k-1})$, for $k\geq 2$.
We decompose $\mu_0$ according to
\[ \mu_0 = \sum_{k\in \NN} \mu_0^k, 
\]
where $\mu_0^k$ is given by $\mu_0^k(E) = \mu_0(E\cap A_k)$,
for every Borel set $E$ in $H$, and $k\in\NN$.

Now, since by construction $\mu(A_k)>0$, we normalize each $\mu_0^k$ to a probability measure by
$\bar\mu_0^k(E)=\mu_0^k(E)/\mu_0(A_k)$, for every Borel set $E$ in $H$.
With the procedure above, we construct Vishik-Fursikov measures
$\bar\rho^k$ on $\Ucal_I^\sharp(R_k)$ with $\Pi_{t_0}\bar\rho^k=\bar\mu_0^k(E)$ 
and satisfying
\eqref{energyestimateforrhoR}. Set
\[ \rho = \sum_{k\in\NN} \mu_0(A_k)\bar\rho^k,
\]
which is clearly a Borel probability measure. Since each $\bar\rho^k$ is
carried by $\Ucal_{[t_0,\infty)}^\sharp(R_k)$, 
then $\rho$ is carried by $\Ucal_{[t_0,\infty)}^\sharp$, 
which proves condition \eqref{rhocarrierforvfm}.

From \eqref{energyestimateforrhoR} for $\bar\rho^k$, we deduce that
\begin{equation}
  \label{energyestimateforrho}
  \int_{\Ucal_{[t_0,\infty)}^\sharp} |\bu(t)|_{L^2}^2\;\rd\rho(\bu)
     \leq \int_{\Ucal_{[t_0,\infty)}^\sharp} |\bu(t_0)|_{L^2}^2\;\rd\rho(\bu)
     + \frac{1}{\nu\lambda_1}\|\bbf\|_{L^\infty(t_0,\infty;H)}^2(t-t_0), \qquad \forall t\geq t_0,
\end{equation}
which proves condition \eqref{rhomeankineticforvfm}. 

It remains to prove condition \eqref{meancontinuityatinitialtimeforvfm}. 
Let $\psi$ be as in \eqref{meancontinuityatinitialtimeforvfm} and assume, without
loss of generality, that $\psi(0)=0$. We write
\[  \int_{\Ucal_{[t_0,\infty)}^\sharp} \psi(|\bu(t)|_{L^2}^2) \;\rd\rho(\bu)
    = \sum_{k\in\NN} \mu_0(A_k) \int_{\Ucal_{[t_0,\infty)}^\sharp} \psi(|\bu(t)|_{L^2}^2) 
    \;\rd\bar\rho^k(\bu),
\]
and note that for a given $K\in \NN$, since each $\bar\rho^k$ is a Vishik-Fursikov measure,
\begin{multline}
  \label{conditioniiiuptoK}
  \lim_{t\rightarrow t_0^+} 
      \sum_{k=1}^K \mu_0(A_k) \int_{\Ucal_{[t_0,\infty)}^\sharp} \psi(|\bu(t)|_{L^2}^2) 
      \;\rd\bar\rho^k(\bu) \\
    = \sum_{k=1}^K \mu_0(A_k) \int_{\Ucal_{[t_0,\infty)}^\sharp} \psi(|\bu(t_0)|_{L^2}^2) \;\rd\bar\rho^k(\bu).
\end{multline}
The remaining terms we estimate using \eqref{energyestimateforrhoR}
for $\bar\rho^k$ and the assumption that $\psi'$ is bounded:
\begin{align*}
  \sum_{k=K+1}^\infty &\mu_0(A_k) \int_{\Ucal_{[t_0,\infty)}^\sharp} \psi(|\bu(t)|_{L^2}^2) 
          \;\rd\bar\rho^k(\bu) \\
    & \leq \sum_{k=K+1}^\infty \mu_0(A_k) (\sup_{r\in\RR} \psi'(r)) \left(
          \int_{\Ucal_{[t_0,\infty)}^\sharp} |\bu(t_0)|_{L^2}^2 \;\rd\bar\rho^k(\bu) \right. \\
     & \qquad \qquad \qquad \qquad \qquad \qquad \qquad \qquad \left.
             + \frac{1}{\nu\lambda_1}\|\bbf\|_{L^\infty(t_0,\infty;H)}^2(t-t_0)\right) \\
    & = \sum_{k=K+1}^\infty (\sup_{r\in\RR} \psi'(r)) \left( \int_H |\bu(t_0)|_{L^2}^2 
             \;\rd\mu_0^k(\bu) 
             + \mu_0(A_k)\frac{1}{\nu\lambda_1}\|\bbf\|_{L^\infty(t_0,\infty;H)}^2(t-t_0)\right) \\
    & = \sup_{r\in\RR} \psi'(r) \left( \int_{H\setminus B_H(R_K)} |\bu(t_0)|_{L^2}^2 
             \;\rd\mu_0^k(\bu) + \right. \\
     & \qquad \qquad \qquad \qquad \qquad \left. \mu_0(H\setminus B_H(R_K))
                 \frac{1}{\nu\lambda_1}\|\bbf\|_{L^\infty(t_0,\infty;H)}^2(t-t_0)\right).
\end{align*}
Since the right hand side above goes to zero uniformly in $t$, as $K$ goes to infinity, and taking 
\eqref{conditioniiiuptoK} into consideration we see that
\[  \int_{\Ucal_{[t_0,\infty)}^\sharp} \psi(|\bu(t)|_{L^2}^2) \;\rd\rho(\bu)
      \rightarrow \int_{\Ucal_{[t_0,\infty)}^\sharp} \psi(|\bu(t_0)|_{L^2}^2) \;\rd\rho(\bu)
\]
as $t\rightarrow t_0^+$, which proves condition \eqref{meancontinuityatinitialtimeforvfm}.
\end{proof}
\medskip

\subsection{Mean energy inequality for Vishik-Fursikov measures}

Notice that in the Definition \ref{defvfmeasure} of a Vishik-Fursikov
measure, there is no explicit condition for some sort of mean
energy inequality. It turns out that this is hidden in the hypothesis
that the measure is concentrated on weak solutions, for which the
individual energy inequality holds. In this direction, we have the
following result.
    
\begin{prop}
  \label{propvsenoughinw}
  Let $I\subset \RR$ be an arbitrary interval.
  Let $\rho$ be a Borel probability measure on $\Ccal_\loc(I,H_\rw)$ and
  suppose $\rho$ is carried by $\Ucal_I^\sharp$, with
  \begin{equation}
    \int_{\Ucal_I^\sharp} |\bu(t)|_{L^2}^2 \;\rd\rho(\bu) < \infty,
  \end{equation}
  for almost all $t\in I$. Then,
  \begin{align}
    \label{rhoenergyislinftyloc}
    & t\mapsto \int_{\Ucal_I^\sharp} |\bu(t)|_{L^2}^2 \;\rd\rho(\bu) \in L_\loc^\infty(I), \\
    \label{rhoenstrophyisl1loc}
    & t\mapsto \int_{\Ucal_I^\sharp} \|\bu(t)\|_{H^1}^2 \;\rd\rho(\bu) \in L_\loc^1(I),
  \end{align}
  and the strengthened mean energy inequality holds on $I$, i.e. there exists a set 
  $I'$ of full measure in $I$ such that for any nonnegative, 
  nondecreasing, continuously-differentiable real-valued function 
  $\psi:[0,\infty)\rightarrow \RR$ with bounded derivative, we have, for all
  $t'\in I'$, and all $t\in I$ with $t\geq t'$, 
  \begin{multline}
    \label{statenergyineqforI}
    \int_{\Ucal_I^\sharp} \left\{ \frac{1}{2}\psi(|\bu(t)|_{L^2}^2) 
       +\nu\int_{t'}^t \psi'(|\bu(s)|_{L^2}^2)\|\bu(s)\|_{H^1}^2 \;\rd s \right\} \;\rd\rho(\bu) \\
          \leq  \int_{\Ucal_I^\sharp} \left\{ \frac{1}{2}\psi(|\bu(t')|_{L^2}^2)
               +  \int_{t'}^t \psi'(|\bu(s)|_{L^2}^2)(\bbf(s),\bu(s))_{L^2} 
               \;\rd s \right\} \;\rd\rho(\bu).
  \end{multline}
\end{prop}

\begin{proof} 
  Let $J\subset I$ be a closed and bounded interval and recall
  the restriction operator $\Pi_J$, which is continuous from $\Ccal_\loc(I,H_\rw)$ 
  to $\Ccal(J,H_\rw)$. We know that $\Pi_J\Ucal_I^\sharp \subset \Ucal_J^\sharp$.  
  We argue that it suffices to show the result for the measure $\rho_J=\Pi_J\rho$, 
  i.e. that
  \begin{multline}
    \label{statenergyineqforJ}
    \int_{\Ucal_J^\sharp} \left\{ \frac{1}{2}\psi(|\bu(t)|_{L^2}^2) 
       +\nu\int_{t'}^t \psi'(|\bu(s)|_{L^2}^2)\|\bu(s)\|_{H^1}^2 \;\rd s \right\} \;\rd\rho_J(\bu) \\
          \leq  \int_{\Ucal_J^\sharp} \left\{ \frac{1}{2}\psi(|\bu(t')|_{L^2}^2)
               +  \int_{t'}^t \psi'(|\bu(s)|_{L^2}^2)(\bbf(s),\bu(s))_{L^2} \;\rd s \right\} 
                 \;\rd\rho_J(\bu)
  \end{multline}
  holds for almost all $t'\in J$ and all $t\in J$, $t\geq t'$, ,
  where $\psi:[0,\infty)\rightarrow \RR$ is an arbitrary nonnegative, 
  nondecreasing, continuously-differentiable real-valued function 
  with bounded derivative. 
  
  In fact, since $\Pi_J^{-1}\Ucal_J^\sharp \supset \Ucal_I^\sharp$ and
  $\rho$ is carried by $\Ucal_I^\sharp$, \eqref{statenergyineqforJ} 
  implies the corresponding result for $\rho$ and $\Ucal_I^\sharp$ but
  still for almost all $t'\in J$ and all $t\in J$ with $t\geq t'$. But since
  $J\subset I$ is an arbitrary closed and bounded interval in $I$,
  the inequality extends to almost all $t'\in I$ and all $t\in I$ with $t\geq t'$,
  which is what we want to prove. Hence, we now only need to prove
  \eqref{statenergyineqforJ}

  Since each $\bu\in \Ucal_J^\sharp$ is weakly continuous and $J$ is compact
  it follows that $\bu$ is uniformly bounded in $H$ on $J$. Hence, we can write
  \[ \Ucal_J^\sharp = \bigcup_{k\in \NN}  \Ucal_J^\sharp(kR_0).
  \]
  Therefore, we first prove the strengthened mean energy inequality in each 
  $\Ucal_J^\sharp(R)$ with $R\geq R_0$, 
  i.e. that there exists $J'$ of full measure on $J$ such that
  \begin{multline*}
    \int_{\Ucal_J^\sharp(R)} \left\{ \frac{1}{2}\psi(|\bu(t)|_{L^2}^2)
      + \nu\int_{t'}^t \psi'(|\bu(s)|_{L^2}^2)\|\bu(s)\|_{H^1}^2 \;\rd s \right\} \;\rd\rho_J(\bu) \\
          \leq  \int_{\Ucal_J^\sharp(R)} \left\{ \frac{1}{2}\psi(|\bu(t')|_{L^2}^2)
               +  \int_{t'}^t \psi'(|\bu(s)|_{L^2}^2)(\bbf(s),\bu(s))_{L^2} \;\rd s \right\} 
                 \;\rd\rho_J(\bu)
  \end{multline*}
  holds for all $t'\in J'$ and all $t\in J$, $t\geq t'$, and for all $\psi$ as in the statement of
  the proposition. 
    
  We have seen in Proposition \ref{ucalitildercompact} that $\Ucal_J^\sharp(R)$ is a compact
  metric space (with the topology inherited from $\Ccal(J,H_\rw)$) 
  and is compactly embedded in $L_\loc^2(J,H)$.
    
  If $\rho(\Ucal_J^\sharp(R))=0$ the mean energy inequality is trivially 
  satisfied, so we need only consider the case in which $0<\rho(\Ucal_J^\sharp(R))\leq 1$,
  which necessarily happens for $R\geq R_0$ sufficiently large. 
   
  Let $\tilde\rho_J$ be the restriction of the measure $\rho_J$ to $\Ucal_J^\sharp(R)$
  normalized to a probability measure 
  (i.e. $\tilde\rho_J(E) = \rho_J(E\cap \Ucal_J^\sharp(R))/\rho_J(\Ucal_J^\sharp(R))$ 
  for any Borel set $E$ in $\Ucal_J^\sharp(R)$). Since $\Ucal_J^\sharp(R)$ is compact
  and separable we can apply the Krein-Milman Theorem to approximate $\tilde\rho_J$
  by finite convex combinations of Dirac measures concentrated on
  weak solutions $\bu_j\in \Ucal_J^\sharp(R)$, $j=1,\ldots, J(n),$ i.e.
  \[ \rho_n = \sum_{j=1}^{J(n)} \theta_j^{(n)} \delta_{\bu_j^{(n)}} 
        \stackrel{*}{\rightharpoonup} \tilde\rho_J, \quad \text{as } n\rightarrow \infty,
  \]
  with $0<\theta_j^{(n)}\leq 1$, $\sum_{j=1}^{J(n)}\theta_j^{(n)}=1$,
  where the convergence is in the weak-star sense:
  \[ \int_{\Ucal_J^\sharp(R)} \varphi(\bu)\;\rd\rho_n(\bu)
              \rightarrow \int_{\Ucal_J^\sharp(R)} \varphi(\bu)\;\rd\tilde\rho_J(\bu),
  \]
  for all $\varphi\in \Ccal(\Ucal_J^\sharp(R))$.

  Since each $\bu_j^{(n)}$ is a Leray-Hopf weak solution on $J$, 
  we have the individual strengthened energy inequality (see \eqref{strengthenedenergyineq})
  \begin{multline*}
     \frac{1}{2}\psi(|\bu_j^{(n)}(t)|_{L^2}^2) 
       + \nu \int_{t'}^t \psi'(|\bu_j^{(n)}(s)|_{L^2}^2)\|\bu_j^{(n)}(s)\|_{H^1}^2 \;\rd s \\
        \leq \frac{1}{2}\psi(|\bu_j^{(n)}(t')|_{L^2}^2) 
        + \int_{t'}^t \psi'(|\bu_j^n(s)|_{L^2}^2)(\bbf(s),\bu_j^{(n)}(s))_{L^2} \;\rd s,
  \end{multline*}
  for almost all $t'\in J$ and all $t\in J$, $t\geq t'$. The set $J_j^{(n)}$ 
  of allowed times $t'$ above depends on the solution (and not on $\psi$), but since we have 
  a countable family of solutions, the intersection $\tilde J = \cap_{j,n} J_j^{(n)}$
  of the allowed times is still of full measure in $J$, hence the energy inequality 
  above holds for almost all $t'$ in $J$, independently of $j,n\in \NN$. 
  Thus, considering the convex combination of the solutions we write
  \begin{multline}
    \label{meanenergyineqonJR}
    \int_{\Ucal_J^\sharp(R)} \left\{ \frac{1}{2}\psi(|\bu(t)|_{L^2}^2) 
        + \nu\int_{t'}^t \psi'(|\bu(s)|_{L^2}^2)\|\bu(s)\|_{H^1}^2 \;\rd s \right\} \;\rd\rho_n(\bu) \\
          \leq  \int_{\Ucal_J^\sharp(R)} \left\{ \frac{1}{2}\psi(|\bu(t')|_{L^2}^2)
               +  \int_{t'}^t \psi'(|\bu(s)|_{L^2}^2)(\bbf(s),\bu(s))_{L^2} \;\rd s \right\} 
               \;\rd\rho_n(\bu),
  \end{multline}
  for all $t'\in \tilde J$ and all $t\in J$ with $t\geq t'$.
  
  Integrate the relation \eqref{meanenergyineqonJR} in $t'$ to find
  \begin{multline}
    \label{meanenergyineqonJRt1t2}
    \int_{t_1}^{t_2}\int_{\Ucal_J^\sharp(R)} \left\{ \frac{1}{2}\psi(|\bu(t)|_{L^2}^2) 
        + \nu\int_{t'}^t \psi'(|\bu(s)|_{L^2}^2)\|\bu(s)\|_{H^1}^2 \;\rd s \right\} 
        \;\rd\rho_n(\bu)\;\rd t' \\
          \leq   \int_{t_1}^{t_2}\int_{\Ucal_J^\sharp(R)} \left\{ \frac{1}{2}\psi(|\bu(t')|_{L^2}^2)
               +  \int_{t'}^t \psi'(|\bu(s)|_{L^2}^2)(\bbf(s),\bu(s))_{L^2} \;\rd s \right\} 
               \;\rd\rho_n(\bu) \;\rd t',
  \end{multline}
  for all $t_1,t_2,t\in J$, $t_1\leq t_2\leq t$.

  Let $t_1,t_2\in J$, $t_1<t_2$, and consider the function
  \[ \varphi(\bu) = \int_{t_1}^{t_2} \psi(|\bu(t')|_{L^2}^2) \;\rd t',
  \]
  which is a well-defined real-valued Borel function on $\Ucal_J^\sharp(R)$, as
  the limit of the continuous functions $\varphi_m(\bu)=\varphi(P_m\bu)$.
  Since $\Ucal_J^\sharp(R)$ is compactly embedded
  in $L_\loc^2(J,H)$ (Proposition \ref{ucalitildercompact}), it also follows that $\varphi(\bu)$ 
  is a continuous
  function in $\Ucal_J^\sharp(R)$. Hence, since $\rho_n$ converges
  weak-star to $\tilde\rho_J$ we have, using the Fubini Theorem, that
  \begin{multline*}
     \int_{t_1}^{t_2} \int_{\Ucal_J^\sharp(R)} \psi(|\bu(t)|_{L^2}^2) \;\rd\rho_n(\bu) \;\rd t'
        = \int_{\Ucal_J^\sharp(R)} \varphi(\bu)  \;\rd\rho_n(\bu) \\
        \rightarrow \int_{\Ucal_J^\sharp(R)} \varphi(\bu)  \;\rd\tilde\rho_J(\bu) 
         = \int_{t_1}^{t_2} \int_{\Ucal_J^\sharp(R)} \psi(|\bu(t)|_{L^2}^2) 
         \;\rd\tilde\rho_J(\bu) \;\rd t'.
  \end{multline*}
  
  Consider now the functions
  \[ \varphi_m(\bu)= \int_{t'}^{t} \psi'(|\bu(s)|_{L^2}^2)\|P_m\bu(s)\|_{H^1}^2 \;\rd s.
  \]
  Since $\Ucal_J^\sharp(R)$ is compactly embedded in $L_\loc^2(J,H)$, it follows 
  that $\varphi(\bu)$ is a continuous function in $\Ucal_J^\sharp(R)$. Then, using
  the Monotone Convergence Theorem and the fact that $\psi'\geq 0$, we deduce that
  \begin{align*}
        \int_{\Ucal_J^\sharp(R)}  & \int_{t'}^{t} \psi'(|\bu(s)|_{L^2}^2)\|\bu(s)\|_{H^1}^2 
          \;\rd s \;\rd\tilde\rho_J(\bu) \\
         & = \lim_{m\rightarrow\infty} \int_{\Ucal_J^\sharp(R)} 
              \int_{t}^{t'} \psi'(|\bu(s)|_{L^2}^2)\|P_m\bu(s)\|_{H^1}^2 \;\rd s \;\rd\tilde\rho_J(\bu) \\
         & = \lim_{m\rightarrow\infty}
              \int_{\Ucal_J^\sharp(R)} \varphi_m(\bu) \;\rd\tilde\rho_J(\bu) \\    
         & = \lim_{m\rightarrow\infty} 
              \lim_{n\rightarrow \infty}  \int_{\Ucal_J^\sharp(R)} \varphi_m(\bu) \;\rd\rho_n(\bu) \\  
         & = \lim_{m\rightarrow\infty} 
              \lim_{n\rightarrow \infty} \int_{\Ucal_J^\sharp(R)}
                \int_{t}^{t'} \psi'(|\bu(s)|_{L^2}^2)\|P_m\bu(s)\|_{H^1}^2 \;\rd s \;\rd\rho_n(\bu)  \\
         & \leq \liminf_{n\rightarrow\infty} \int_{\Ucal_J^\sharp(R)}
           \int_{t}^{t'} \psi'(|\bu(s)|_{L^2}^2)\|\bu(s)\|_{H^1}^2 \;\rd s \;\rd\rho_n(\bu),
  \end{align*}
  for any $t,t'\in J$, $t'\leq t$. Similarly,
  \begin{multline*}
    \frac{1}{2}\int_{\Ucal_J^\sharp(R)}\psi(|\bu(t)|_{L^2}^2) \;\rd\tilde\rho_J(\bu)
        = \lim_{m\rightarrow \infty} \frac{1}{2}\int_{\Ucal_J^\sharp(R)}\psi(|P_m\bu(t)|_{L^2}^2) 
             \;\rd\tilde\rho_J(\bu) \\
        = \lim_{m\rightarrow \infty} \lim_{n\rightarrow \infty} 
             \frac{1}{2}\int_{\Ucal_J^\sharp(R)}\psi(|P_m\bu(t)|_{L^2}^2) \;\rd\rho_n(\bu)
        \leq \liminf_{n\rightarrow \infty} 
             \frac{1}{2}\int_{\Ucal_J^\sharp(R)}\psi(|\bu(t)|_{L^2}^2) \;\rd\rho_n(\bu).
  \end{multline*}
  On the other hand, since the function
  \[ \bu \mapsto \int_{t'}^t \psi'(|\bu(s)|_{L^2}^2)(\bbf(s),\bu(s))_{L^2} \;\rd s 
  \]
  is continuous on $\Ucal_J^\sharp(R)$, it follows that
  \begin{multline*}
    \int_{\Ucal_J^\sharp(R)} \int_{t'}^t \psi'(|\bu(s)|_{L^2}^2)(\bbf(s),\bu(s))_{L^2} 
      \;\rd s \;\rd \tilde\rho_J(\bu) \\
       = \lim_{n\rightarrow\infty} \int_{\Ucal_J^\sharp(R)} \int_{t'}^t \psi'(|\bu(s)|_{L^2}^2)
          (\bbf(s),\bu(s))_{L^2} \;\rd s \;\rd\rho_n(\bu).
  \end{multline*}
  
  Thus, passing to the limit in the mean energy equation for $\rho_n$
  we find that
  \begin{multline}
    \label{meanenergyineqonJRt1t2rhoJ}
    \int_{t_1}^{t_2}\int_{\Ucal_J^\sharp(R)} \left\{ \frac{1}{2}\psi(|\bu(t)|_{L^2}^2) 
       + \nu\int_{t'}^t \psi'(|\bu(s)|_{L^2}^2)\|\bu(s)\|_{H^1}^2 \;\rd s \right\} 
       \;\rd\tilde\rho_J(\bu)\;\rd t'\\
          \leq  \int_{t_1}^{t_2}\int_{\Ucal_J^\sharp(R)} \left\{ \frac{1}{2}\psi(|\bu(t')|_{L^2}^2)
               +  \int_{t'}^t \psi'(|\bu(s)|_{L^2}^2)(\bbf(s),\bu(s))_{L^2} \;\rd s \right\} 
               \;\rd\tilde\rho_J(\bu) \;\rd t',
  \end{multline}
  for all $t_1,t_2,t\in J$ with $t_1<t_2<t$.   
  
  The aim now is to divide \eqref{meanenergyineqonJRt1t2rhoJ} by $(t_2-t_1)$
  and take the limit as $t_2\rightarrow t_1^+$. The integrand (with respect to $t'$) 
  of the first term on the left hand side of \eqref{meanenergyineqonJRt1t2rhoJ} 
  is independent of $t'$, while the integrand of the second term in the left hand side and that 
  of the second term in the right hand side are continuous in $t'$ and present no further
  difficulties either. The only delicate term is the first one on the right hand side.
  
  For that, we consider a countable dense set $\{\psi_\ell\}_\ell$, 
  with respect to the uniform topology of $\Ccal([0,\infty),\RR)$, in the space
  of nonnegative, nondecreasing, continuously-differentiable real-valued function 
  with bounded derivative. For each $\ell$, the function
  \[  t \mapsto \int_{\Ucal_J^\sharp(R)} \frac{1}{2}\psi_\ell(|\bu(t)|_{L^2}^2) \;\rd\tilde\rho_J(\bu)
  \]
  is integrable on $J$, and hence, by the Lebesgue Differentiation Theorem, there exists
  a set $J_\ell$ of full measure in $J$ such that for all $t_1\in J_\ell$,
  \[ \frac{1}{t_2-t_1} \int_{t_1}^{t_2} \int_{\Ucal_J^\sharp(R)} 
       \frac{1}{2}\psi_\ell(|\bu(t)|_{L^2}^2) \;\rd\tilde\rho_J(\bu) 
       \rightarrow \int_{\Ucal_J^\sharp(R)} \frac{1}{2}\psi_\ell(|\bu(t_1)|_{L^2}^2) 
       \;\rd\tilde\rho_J(\bu),
       \quad \text{as } t_2\rightarrow t_1^+.
  \]
  Setting $J'=\bigcap_\ell J_\ell$, we have that $J'$ is of full
  measure and the convergence above holds for all $t_1\in J$ and all $\ell\in \NN$.
  Thus, dividing the expression \eqref{meanenergyineqonJRt1t2rhoJ} by $t_2-t_1$,
  with $\psi=\psi_\ell$ for each $\ell$, and passing to the limit as $t_2\rightarrow t_1^+$,
  we find that
  \begin{multline}
    \label{meanenergyineqonJRrhoJpsiell}
    \int_{\Ucal_J^\sharp(R)} \left\{ \frac{1}{2}\psi_\ell(|\bu(t)|_{L^2}^2) 
       + \nu\int_{t_1}^t \psi_\ell'(|\bu(s)|_{L^2}^2)\|\bu(s)\|_{H^1}^2 \;\rd s \right\} 
       \;\rd\tilde\rho_J(\bu)\\
          \leq  \int_{\Ucal_J^\sharp(R)} \left\{ \frac{1}{2}\psi_\ell(|\bu(t_1)|_{L^2}^2)
               +  \int_{t_1}^t \psi_\ell'(|\bu(s)|_{L^2}^2)(\bbf(s),\bu(s))_{L^2} \;\rd s \right\} 
                 \;\rd\tilde\rho_J(\bu),
  \end{multline}
  for all $t_1\in J'$, all $t\in J$ with $t_1<t$, and all $\ell\in \NN$. Using the
  density of $\{\psi_\ell\}_{\ell\in\NN}$ and renaming $t_1$ as $t'$, we find
  \begin{multline}
    \label{meanenergyineqonJRrhoJ}
    \int_{\Ucal_J^\sharp(R)} \left\{ \frac{1}{2}\psi(|\bu(t)|_{L^2}^2) 
       + \nu\int_{t'}^t \psi'(|\bu(s)|_{L^2}^2)\|\bu(s)\|_{H^1}^2 \;\rd s \right\} \;\rd\tilde\rho_J(\bu)\\
          \leq  \int_{\Ucal_J^\sharp(R)} \left\{ \frac{1}{2}\psi(|\bu(t')|_{L^2}^2)
               +  \int_{t'}^t \psi'(|\bu(s)|_{L^2}^2)(\bbf(s),\bu(s))_{L^2} \;\rd s \right\} 
               \;\rd\tilde\rho_J(\bu),
  \end{multline}
  for all $t'\in J'$, all $t\in J$ with $t_1<t$, and all $\psi$. This proves the strengthened mean 
  energy inequality in $\Ucal_J^\sharp(R)$
   
  Applying now the Cauchy-Schwarz, Poincar\'e, and Young inequalities to
  the forcing term in the mean energy inequality \eqref{meanenergyineqonJRrhoJ} 
  with $\psi(r)=r$, we find that
   \begin{multline*}
    \int_{\Ucal_J^\sharp(R)} \left\{ |\bu(t)|_{L^2}^2 +\nu\int_{t'}^t \|\bu(s)\|_{H^1}^2 \;\rd s 
             \right\} \;\rd\tilde\rho_J(\bu) \\
          \leq  \int_{\Ucal_J^\sharp(R)}  |\bu(t')|_{L^2}^2 \;\rd\tilde\rho_J(\bu)
               +  \frac{1}{\nu\lambda_1}(t-t')\|\bbf\|_{L^\infty(I;H)}^2 \\
          \leq  \int_{\Ucal_J^\sharp}  |\bu(t')|_{L^2}^2 \;\rd\tilde\rho_J(\bu)
               +  \frac{1}{\nu\lambda_1}(t-t')\|\bbf\|_{L^\infty(I;H)}^2,
  \end{multline*}
  for almost all $t'\in J$ and all $t\in J$ with $t\geq t'$.
  
  Using the definition of $\tilde\rho_J$ we can rewrite the inequality 
  above as
  \begin{multline*}
    \frac{1}{\rho_J(\Ucal_J^\sharp(R))}
       \int_{\Ucal_J^\sharp(R)} \left\{ |\bu(t)|_{L^2}^2 +\nu\int_{t'}^t \|\bu(s)\|_{H^1}^2 \;\rd s 
             \right\} \;\rd\rho_J(\bu) \\
          \leq   \frac{1}{\rho_J(\Ucal_J^\sharp(R))}
              \int_{\Ucal_J^\sharp(R)}  |\bu(t')|_{L^2}^2 \;\rd\rho_J(\bu)
               +  \frac{1}{\nu\lambda_1}(t-t')\|\bbf\|_{L^\infty(I;H)}^2 ,
  \end{multline*}

  Applying the Monotone Convergence Theorem to the left
  hand side of the inequality above we find, by taking the limit
  $R\rightarrow \infty$, that
   \begin{multline*} 
     \int_{\Ucal_I^\sharp} \left\{ |\bu(t)|_{L^2}^2 +\nu\int_{t'}^t \|\bu(s)\|_{H^1}^2 \;\rd s 
             \right\} \;\rd\rho_J(\bu) \\
          \leq  \int_{\Ucal_J^\sharp}  |\bu(t')|_{L^2}^2 \;\rd\rho_J(\bu)
               +  \frac{1}{\nu\lambda_1}(t-t')\|\bbf\|_{L^\infty(I;H)}^2,
  \end{multline*}
  for almost all $t'\in J$ and all $t\in J$ with $t\geq t'$.
  
  Since we are assuming that the mean kinetic energy is finite for almost every
  $t'$ in $J$, we consider $t'$ which is ``good'' for the energy
  inequality and for the corresponding finite mean kinetic energy, and such
  set of $t'$ is still of full measure in $I$. Then, we find from above that
  \[ t\rightarrow \int_{\Ucal_J^\sharp} |\bu(t)|_{L^2}^2 \;\rd\rho_J(\bu)
       \in L_\loc^\infty(J^\circ)
  \]
  and is finite everywhere on $J$, and
  \[ t\rightarrow \int_{\Ucal_I^\sharp} \|\bu(t)\|_{H^1}^2 \;\rd\rho(\bu) 
       \in L_\loc^1(J^\circ),
  \]
  where $J^\circ$ is the interior of the interval $J$. Since these hold for any $J\subset I$
  closed and bounded, they imply \eqref{rhoenergyislinftyloc}
  and \eqref{rhoenstrophyisl1loc}.
  
  Now we go back to the mean energy inequality \eqref{meanenergyineqonJRrhoJ}
  on $\Ucal_J^\sharp(kR_0)$ and pass to the limit as $k\rightarrow \infty$ using the 
  the Lebesgue Dominated Convergence Theorem to obtain the mean energy inequality
  on the whole space $\Ucal_J^\sharp$:
  \begin{multline*}
    \int_{\Ucal_J^\sharp} \left\{ \frac{1}{2}\psi(|\bu(t)|_{L^2}^2) 
      + \nu\int_{t'}^t \psi'(|\bu(t)|_{L^2}^2)\|\bu(s)\|_{H^1}^2 \;\rd s \right\} \;\rd\rho_J(\bu) \\
          \leq  \int_{\Ucal_J^\sharp} \left\{ \frac{1}{2}\psi(|\bu(t')|_{L^2}^2)
               +  \int_{t'}^t \psi'(|\bu(t)|_{L^2}^2)(\bbf(s),\bu(s))_{L^2} \;\rd s \right\} 
                 \;\rd\rho_J(\bu),
  \end{multline*}
  for almost all $t'\in J$ and all $t\in J$ with $t\geq t'$. This completes the proof.
\end{proof}
\medskip

\begin{rmk}
  Notice that the idea of the proof of Theorem \ref{propvsenoughinw} is to reduce the
  proof of the strengthened mean energy inequality from $\Ucal_I^\sharp$ to the case
  $\Ucal_J^\sharp(R)$, for $R\geq R_0$ and $J\subset I$ compact, and then, for 
  a Vishik-Fursikov measure carried by the compact Hausdorff space $\Ucal_J^\sharp(R)$, 
  the idea is to use the Krein-Milman Theorem to approximate this measure by 
  a convex combination 
  of Dirac measures concentrated on individual weak solutions for which the 
  strengthened mean energy inequality holds. This approximation of a Vishik-Fursikov
  measure by convex combinations of Dirac measures will be further explored in
  Section \ref{characvfss} to characterize statistical solutions which are 
  projections of Vishik-Fursikov measures. 
\end{rmk}

\begin{prop}
  \label{propvsenoughinw2}
  Let $I\subset \RR$ be an interval closed and bounded on the left 
  with left end point $t_0$. Let $\rho$ be a Borel probability 
  measure on $\Ccal_\loc(I,H_\rw)$ carried by $\Ucal_I^\sharp$. Suppose that
  \begin{equation}
    \int_{\Ucal_I^\sharp} |\bu(t)|_{L^2}^2 \;\rd\rho(\bu) < \infty,
  \end{equation}
  for almost all $t\in I$ and
  \begin{equation}
    \label{continuitywithpsi}
    \lim_{t\rightarrow t_0^+}  \int_{\Ucal_I^\sharp} \psi(|\bu(t)|_{L^2}^2) \;\rd\rho(\bu) 
      =  \int_{\Ucal_I^\sharp} \psi(|\bu(t_0)|_{L^2}^2) \;\rd\rho(\bu) <\infty,
  \end{equation}
  for every function $\psi$ as in Proposition \ref{propvsenoughinw}.
  Then,
  \begin{equation}
    t\mapsto \int_{\Ucal_I^\sharp} |\bu(t)|_{L^2}^2 \;\rd\rho(\bu) \in L_\loc^\infty(I),
  \end{equation}
  and the strengthened mean energy inequality \eqref{statenergyineqforI} holds including 
  at time $t'=t_0$.
\end{prop}

\begin{proof}
  In view of Proposition \ref{propvsenoughinw} we only need to prove that
  the strengthened mean energy inequality holds also at $t'=t_0$.
  But this follows trivially from the corresponding inequality for $t'>t_0$
  and taking the limit as $t'\rightarrow t_0$ using condition \eqref{continuitywithpsi}.
\end{proof}
\medskip

Since any Vishik-Fursikov measure satisfies, by definition, the hypotheses in 
Propositions \ref{propvsenoughinw} and \ref{propvsenoughinw2}, we have the following 
corollary, which we state as a theorem.

\begin{thm}
  \label{thmvfmsatisfiesenergyine}
    Let $\rho$ be a Vishik-Fursikov measure over an interval $I\subset \RR$.
    Then, the strengthened mean energy inequality
    \eqref{statenergyineqforI} holds on $I$.
    If $I$ is closed and bounded on the left, then \eqref{statenergyineqforI} holds in particular 
    for $t'$ being the left end point of $I$.
    \qed
\end{thm}
\medskip

\begin{rmk}
  Let $I\subset \RR$ be an arbitrary interval. For each $\bu\in \Ccal_\loc(I,H_\rw)$ and $t\in I$,
  define 
  \[ h_+(\bu,t)=\limsup_{\tau\rightarrow t^+} |\bu(\tau)|_{L^2}.
  \]
  It is not difficult to see that the strengthened energy inequality 
  \eqref{strengthenedenergyineq} 
  implies that for any $\bu\in\Ucal_I^\sharp$ and for any $\psi$ as in 
  \eqref{strengthenedenergyineq},
  \begin{multline}
    \frac{1}{2}\psi(h_+^2(\bu,t)) + \nu\int_{t_1}^t \psi'(|\bu(t)|_{L^2}^2)\|\bu(s)\|_{H^1}^2 \;\rd s \\
          \leq  \frac{1}{2}\psi(h_+^2(\bu,t_1)) +  \int_{t_1}^t \psi'(|\bu(t)|_{L^2}^2)
            (\bbf(s),\bu(s))_{L^2} \;\rd s,
  \end{multline}
  for all $t,t_1\in I$ with $t\geq t_1$. Similarly, for a Vishik-Fursikov measure $\rho$ over
  $I$, we find
  \begin{multline}
    \label{meanstrengthenedlimsup}
    \int_{\Ucal_I^\sharp} \left\{\frac{1}{2}\psi(h_+^2(\bu,t)) 
       + \nu\int_{t_1}^t \psi'(|\bu(t)|_{L^2}^2)\|\bu(s)\|_{H^1}^2 \;\rd s \right\}\rd\rho(\bu) \\
          \leq \int_{\Ucal_I^\sharp} \left\{ \frac{1}{2}\psi(h_+^2(\bu,t_1)) 
            +  \int_{t_1}^t \psi'(|\bu(t)|_{L^2}^2)(\bbf(s),\bu(s))_{L^2} \;\rd s \right\}\rd\rho(\bu),
  \end{multline}  
  for all $t,t_1\in J$ with $t\geq t_1$. The advantage of introducing the term $h_+(\bu,t)$ is that, with this term, the inequalities hold everywhere, not just almost everywhere and in a set depending on the solution. We do not use this inequalities here, but it might be useful when considering a collection of weak solutions or statistical solutions.
\end{rmk}

\subsection{Definition and existence of time-dependent Vishik-Fursikov statistical solutions}

A Borel probability measure $\rho$ on $\Ccal_\loc(I,H_\rw)$ induces a family of 
time-dependent Borel probability measures $\{\rho_t\}_{t\in I}$ in the phase space $H$
defined by the projections $\rho_t=\Pi_t\rho,$ so that
\begin{equation}
  \label{2.15a}
  \int_H \varphi(\bu) \;\rd\rho_t(\bu)
    = \int_{\Ccal_\loc(I,H_\rw)} \varphi(\bv(t)) \;\rd\rho(\bv),
    \qquad \forall t\in I,
\end{equation}
for every $\varphi$ which belongs to $L^1(\rho_t)$ for any $t\in I$; in particular for any
$\varphi$ in $\Ccal_\rb(H_\rw)$.

The following proposition gives the first relation between the statistical solutions of the 
NSE in the sense of Definition \ref{deftimedependetstatisticalsolution}
and the Vishik-Fursikov measures in Definition \ref{defvfmeasure}
(see \cite{vishikfursikov88,fmrt2001a}).

\begin{thm}
  \label{defvfsshalfclosedinterval}
  Let $I\subset \RR$ be an interval closed and bounded on the left, with the left
  end point denoted $t_0$. If $\rho$ is a Vishik-Fursikov 
  measure over $I$, then the family of projections $\{\rho_t\}_{t\in I}$ is a statistical 
  solution with initial data $\rho_{t_0}$ (in the sense of Definition 
  \ref{deftimedependetstatisticalsolution}).
\end{thm}

\begin{proof} 
  We need to check the conditions in the 
  Definition \ref{deftimedependetstatisticalsolution} of statistical solution.
  Conditions \eqref{ssenergy}, \eqref{ssenstrophy}, \eqref{ssstmeaneneineq}, and 
  \eqref{ssinitialtime} follow immediately from the definition 
  of $\{\rho_t\}_{t\in I}$ as the projections of $\rho$ and the properties of 
  $\rho$ given in Proposition \ref{propvsenoughinw2} and in the Definition 
  \ref{defvfmeasure} of Vishik-Fursikov. The only remark worth mentioning
  is that the projection \eqref{2.15a} from $\rho$ to $\rho_t$ is initially valid for 
  $\varphi\in \Ccal_\rb(H_\rw)$, but this can be easily extended to 
  $\varphi(\bu)=|\bu|_{L^2}^2$, and $\varphi(\bu)=\|\bu\|_{H^1}^2$, and so on, by
  using \eqref{integrabilityinducedmeasures}.
    
  Condition \eqref{ssphimeas} follows from the fact that
   \[  t\mapsto \int_H \Phi(\bu) \;\rd\rho_t(\bu) = \int_{\Ccal_\loc(I;H_\rw)} \Phi(\bu(t))\;\rd\rho(\bu)
   \]
   is the pointwise (in time) limit (thanks to the Lebesgue Dominate Convergence Theorem)
   of the continuous function
   \[ t\mapsto \int_{\Ccal_\loc(I;H_\rw)} \Phi(P_m\bu(t))\;\rd\rho(\bu),
   \]
   as $m\rightarrow \infty$, where the $P_m$ are the Galerkin projectors.
   
   It remains to prove condition \eqref{ssliouville}, which is the Liouville-type equation.
   Let $\Phi(\bu)=\psi((\bu,\bv_1),\ldots,(\bu,\bv_k))$ be a cylindrical test function
   as in Definition \ref{deftestfunction}. For $\bu\in \Ucal_I^\sharp$,
   we have $t\mapsto (\bu(t),\bv_j)$ absolutely continuous on $I$ with
   \[ \frac{\rd}{\rd t} (\bu(t),\bv_j) = \dual{\bF(\bu(t)),\bv_j}_{V',V}\in L_\loc^{4/3}(I),
   \]
   where $\bF(\bu) = \bbf - \nu A\bu - B(\bu,\bu)$. Thus, since $\phi$ is continuously
   differentiable, $t\mapsto \Phi(\bu(t))$ is also absolutely continuous on $I$, with
   \begin{align*}
     \frac{\rd}{\rd t} \Phi(\bu(t)) 
       & = \sum_{j=1}^k \partial_j \phi((\bu,\bv_1),\ldots,(\bu,\bv_k)) \frac{\rd}{\rd t}(\bu(t),\bv_j) \\
       & = \sum_{j=1}^k \partial_j \phi((\bu,\bv_1),\ldots,(\bu,\bv_k)) 
         \dual{\bF(\bu(t)),\bv_j}_{V',V} \\
       & = \dual{\bF(\bu(t)),\Phi'(\bu(t))}_{V',V} \in L_\loc^{4/3}(I).
   \end{align*}
   Hence,
   \[ \Phi(\bu(t)) = \Phi(\bu(t')) + \int_{t'}^t \dual{\bF(\bu(s)),\Phi'(\bu(s))}_{V',V} \;\rd s,
   \]
   for all $t',t\in I$. Since $\rho$ is carried by $\Ucal_I^\sharp$, the previous 
   relation holds $\rho$-almost everywhere, so that, upon integration,
   \begin{multline*}
     \int_{\Ccal_\loc(I,H_\rw)}\Phi(\bu(t)) \;\rd\rho(\bu) \\
        = \int_{\Ccal_\loc(I,H_\rw)}\Phi(\bu(t')) \;\rd\rho(\bu)
          + \int_{\Ccal_\loc(I,H_\rw)}\int_{t'}^t \dual{\bF(\bu(s)),\Phi'(\bu(s))}_{V',V} \;\rd s\rd\rho(\bu),
   \end{multline*}
   for all $t',t\in I$. The second integrand in the right hand side belongs to $L^{4/3}$ 
   (in $t$ and $\bu$ for the Lebesgue and $\rho$ measures, respectively) so we
   apply the Fubini Theorem to write
   \begin{multline*}
     \int_{\Ccal_\loc(I,H_\rw)}\Phi(\bu(t)) \;\rd\rho(\bu) \\
        = \int_{\Ccal_\loc(I,H_\rw)}\Phi(\bu(t')) \;\rd\rho(\bu)
          + \int_{t'}^t\int_{\Ccal_\loc(I,H_\rw)} \dual{\bF(\bu(s)),\Phi'(\bu(s))}_{V',V} \;\rd\rho(\bu)\rd s.
   \end{multline*}
   Since $\bu\mapsto \Phi(\bu(t))$ is continuous on $\Ccal_\loc(I,H_\rw)$, we have
   \[ \int_{\Ccal_\loc(I,H_\rw)}\Phi(\bu(t)) \;\rd\rho(\bu) = \int_H \Phi(\bu)\;\rd\rho_t(\bu).
   \]
   Since the map $\bu\mapsto \dual{\bF(P_m\bu(s)),\Phi'(\bu(s))}_{V',V}$ is continuous on 
   $\Ccal_\loc(I,H_\rw)$,
   and using the Lebesgue Dominated Convergence Theorem, we find also that
   \begin{align*}
     \int_{\Ccal_\loc(I,H_\rw)} & \dual{\bF(\bu(s)),\Phi'(\bu(s))}_{V',V} \;\rd\rho(\bu) \\
       & = \lim_{m\rightarrow\infty}\int_{\Ccal_\loc(I,H_\rw)} 
         \dual{\bF(P_m\bu(s)),\Phi'(\bu(s))}_{V',V} \;\rd\rho(\bu) \\
       & = \lim_{m\rightarrow\infty}\int_H \dual{\bF(P_m\bu),\Phi'(\bu)}_{V',V} \;\rd\rho_s(\bu) \\
       & = \int_H \dual{\bF(\bu),\Phi'(\bu)}_{V',V} \;\rd\rho_s(\bu).
   \end{align*}
   Thus,
   \[ \int_H\Phi(\bu) \;\rd\rho_t(\bu) = \int_H\Phi(\bu) \;\rd\rho_{t'}(\bu)
          + \int_{t'}^t\int_H \dual{\bF(\bu),\Phi'(\bu)}_{V',V} \;\rd\rho_s(\bu)\rd s,
   \]
   for all $t',t\in I$. This completes the proof. 
\end{proof}
\medskip

There is also the corresponding result for an interval open on the left.

\begin{thm}
  \label{defvfssopeninterval}
  Let $I\subset \RR$ be an interval open on the left.
  If $\rho$ is a Vishik-Fursikov measure over the
  interval $I$, then $\{\rho_t\}_{t\in I}$ is a statistical
  solution on $I$ (in the sense of Definition 
  \ref{deftimedependetstatisticalsolution}).
\end{thm}

\begin{proof}
 Similarly to the proof of Theorem \ref{defvfsshalfclosedinterval}, except that we need not
 check condition \eqref{ssinitialtime} of Definition \ref{deftimedependetstatisticalsolution}.
\end{proof}
\medskip

A statistical solution in the sense of Definition \ref{deftimedependetstatisticalsolution}
obtained from a Vishik-Fursikov measure is called a Vishik-Fursikov statistical
solution. More precisely:

\begin{defs}
  \label{defvfss}
  Let $I\subset \RR$ be an arbitrary interval. 
  A Vishik-Fursikov statistical solution of the Navier-Stokes equations
  over $I$ is a statistical solution 
  $\{\rho_t\}_{t\in I}$ such that $\rho_t=\Pi_t\rho$, for all $t\in I$,
  for some Vishik-Fursikov measure $\rho$ over the interval $I$.
\end{defs}

The following existence result follows immediately from Theorems
\ref{thmexistencevfmeasure} and \ref{defvfsshalfclosedinterval}.

\begin{thm}
  \label{thmexistencevfss}
  Let $t_0\in \RR$ and let $\mu_0$ be a Borel probability 
  measure on $H$ satisfying
  \[  \int_H |\bu|_{L^2}^2\;\rd\mu_0(\bu) < \infty.
  \]
  Then, there exists a Vishik-Fursikov statistical solution
  $\{\rho_t\}_{t\geq t_0}$ over the interval $I=[t_0,\infty)$ 
  satisfying $\rho_{t_0}=\mu_0$.
  \qed
\end{thm}

\section{Further properties of the Vishik-Fursikov statistical solutions}

\subsection{On the carrier of Vishik-Fursikov measures}

In this section we are interested in the case in which the interval
$I$ is closed and bounded on the left. The case in which $I$ is
open on the left is void of interested for what we are about to discuss since in this case
the spaces $\Ucal_I^\sharp$ and $\Ucal_I$ are known to be equal. 

We denote the left end point of $I$ by $t_0$.
We have defined the Vishik-Fursikov measures 
as measures carried by $\Ucal_I^\sharp$ and satisfying in particular
a certain strengthened continuity at the initial time $t_0$ 
of the mean kinetic energy (condition \eqref{meancontinuityatinitialtimeforvfm} in 
Definition \ref{defvfmeasure}), although the
corresponding energy inequality for the individual solutions
in $\Ucal_I^\sharp$ may not be valid at $t'=t_0$. The aim of this
section is to prove that a Vishik-Fursikov measure is in
fact carried by $\Ucal_I$, hence it is carried by the
individual solutions for which the kinetic energy is continuous at the origin.
Notice that since $H$ is a Hilbert space and since the kinetic energy is
essentially the square of the norm in $H$, then the continuity of the kinetic
energy for an individual solution is equivalent to strong continuity in $H$ (bearing in mind that weak solutions in $\Ucal_I^\sharp$ are weakly continuous also at $t=t_0$).
So in essence the result we are about to prove says that the continuity of the
kinetic energy at time $t=t_0$ in the mean implies the strong continuity of the 
individual solutions at times $t=t_0$ almost everywhere.

\begin{thm}
  \label{vfsscarriedbyus}
  Let $I\subset \RR$ be an arbitrary interval. 
  Let $\rho$ be a Vishik-Fursikov measure over $I$. 
  Then $\rho$ is carried by $\Ucal_I$.
\end{thm}

\begin{proof} 
As mentioned above, we only need to consider the case in
which $I$ is closed and bounded on the left, since otherwise both
spaces $\Ucal_I^\sharp$ and $\Ucal_I$ are known to be equal.
Consider then such an interval $I$ and let $t_0$ be the left end
point of $I$.

Consider the function
\[ \Psi(\bu,t) = \frac{1}{t-t_0} \int_{t_0}^t (|\bu(s)|_{L^2}^2 -|\bu(t_0)|_{L^2}^2 )\;\rd s,
      \quad \text{for all } t\in I, \; t>t_0,
\]
which was defined in Lemma \ref{lemphicondition} and proved to
be a Borel function on $\Ccal_\loc(I,H_\rw)\times (I\setminus\{t_0\})$. 
Consider also the function
\[ \Lambda(\bu) = \liminf_{t\rightarrow t_0^+} \Psi(\bu,t),
\]
which, from Lemma \ref{lemphicondition}, is a Borel function in $\bu\in \Ucal_I^\sharp$,
with  $\Lambda(\bu)\geq 0$ for all $\bu\in \Ucal_I^\sharp$,
and such that $\Lambda(\bu) = 0$ if and only if $\bu\in \Ucal_I$.
Therefore, in order to show that $\rho$ is carried by $\Ucal_I$
it suffices to show that $\Lambda(\bu)=0$ for $\rho$-almost all $\bu$
in $\Ucal_I^\sharp$. Since $\Lambda(\bu)\geq 0$ for all $\bu\in \Ucal_I^\sharp$,
it suffices to show that
\begin{equation}
  \int_{\Ucal_I^\sharp} \Lambda(\bu) \;\rd\rho(\bu) = 0.
\end{equation}
For that purpose, let us recall the mean energy inequality (with $\psi(r)=r$ in the
definition) satisfied by the Vishik-Fursikov measure starting at $t'=t_0$:
\begin{multline*}
  \int_{\Ucal_I^\sharp} \left\{ \frac{1}{2}|\bu(s)|_{L^2}^2 +\nu\int_{t_0}^s \|\bu(\tau)\|_{H^1}^2 
    \;\rd \tau
            \right\} \;\rd\rho(\bu) \\
      \leq \int_{\Ucal_I^\sharp} \left\{ \frac{1}{2}|\bu(t_0)|_{L^2}^2 
              + \int_{t_0}^s (\bbf(\tau),\bu(\tau))_{L^2} \;\rd \tau \right\} \;\rd\rho(\bu),      
\end{multline*}
for all $s\in I$. Using the Cauchy-Schwarz, Poincar\'e, and Young inequalities on 
the forcing term we find that
\[  \int_{\Ucal_I^\sharp} \left\{ |\bu(s)|_{L^2}^2  - |\bu(t_0)|_{L^2}^2 
     +\nu\int_{t_0}^s \|\bu(\tau)\|_{H^1}^2 \;\rd \tau \right\}\;\rd\rho(\bu)
      \leq  \frac{1}{\nu\lambda_1} (s-t_0)\|\bbf\|_{L^\infty(I;H)}^2.
\]

Taking the time average with respect to $s$ in the interval from $t_0$ to $t$ and 
discarding the viscous term we find that
\[ \int_{\Ucal_I^\sharp} \Psi(\bu,t) \;\rd\rho(\bu) \leq
      \frac{1}{2\nu\lambda_1}(t-t_0)\|\bbf\|_{L^\infty(I;H)}^2, 
\]
for all $t\in I\setminus \{t_0\}$. Let $t$ goes to $t_0$ to obtain
\[ \limsup_{t\rightarrow t_0^+} \int_{\Ucal_I^\sharp} \Psi(\bu,t) \;\rd\rho(\bu)
      \leq 0.
\] 
Now, since $\Psi(\bu,t)\geq -|\bu(t_0)|_{L^2}^2$ for all $t\in I$ and the function
$\bu\mapsto -|\bu(t_0)|_{L^2}^2$
is $\rho$-integrable we may apply Fatou's Lemma
to deduce that 
\begin{multline*}
   \int_{\Ucal_I^\sharp(R)} \Lambda(\bu)\;\rd\rho(\bu)
     = \int_{\Ucal_I^\sharp(R)} \liminf_{t\rightarrow t_0^+} \Psi(\bu,t)\;\rd\rho(\bu) \\
     \leq \liminf_{t\rightarrow t_0^+} \int_{\Ucal_I^\sharp(R)} \Psi(\bu,t) \;\rd\rho(\bu)
     \leq \limsup_{t\rightarrow t_0^+} \int_{\Ucal_I^\sharp} \Psi(\bu,t) \;\rd\rho(\bu)
      \leq 0.
\end{multline*}
Since $\Lambda(\bu)\geq 0$ for all $\bu\in \Ucal_I^\sharp$ we deduce from 
the previous inequality that
$\Lambda(\bu)=0$ for $\rho$-almost all $\bu\in \Ucal_I^\sharp$,
which means that $\rho$ is carried by $\Ucal_I$, and the proof is complete.
\end{proof}
\medskip

\subsection{Accretion property for Vishik-Fursikov statistical solutions}

The evolution of the measure of an ensemble of initial data 
will be studied in the forthcoming work \cite{frtssp2}. In that work we will be mainly
concerned with stationary statistical solution, but a particular result about accretion
will be first obtained for time-dependent Vishik-Fursikov statistical solutions.
For that reason, we will mention this result in the following remark.

\begin{rmk}
  Consider the multi-valued evolution operator defined as follows:
  Given a set $E$ in $H$ and $t\geq 0$, we denote by $\Sigma_t E$
  the set of all points $\bw\in H$, such that $\bw=\bu(t)$ and $\bu$ is in $\Ucal_{[0,\infty)}$ 
  with initial condition $\bu(0)\in  E$. As it is proved in \cite{frtssp2}, given any Borel
  set $E$ and any $t\geq 0$, the set $\Sigma_t E$ is universally
  measurable in $H$, i.e. it is measurable with respect to the Lebesgue
  completion of any Borel measure in $H$. Then, we prove in \cite{frtssp2} that
  if $\{\rho_t\}_{t\geq 0}$ is a Vishik-Fursikov statistical solution over $[0,\infty)$,
  we have that $\{\rho_t\}_{t\geq 0}$ is accretive with respect to the family 
  $\{\Sigma_t\}_{t\geq 0}$ in the sense that
  \[ \rho_t(\Sigma_t E) \geq \rho_0(E),
  \]
  for all Borel sets $E$ in $H$ and all $t\geq 0$.
\end{rmk}

\subsection{Characterization of the Vishik-Fursikov statistical solutions}
\label{characvfss}
After Theorems \ref{defvfsshalfclosedinterval} and \ref{defvfssopeninterval}
one can raise the following question:  Given 
a statistical solution $\{\mu_t\}_{t\in I}$ in the sense of Definition 
\ref{deftimedependetstatisticalsolution}, for some interval 
$I\subset \RR$, is there a Vishik-Fursikov probability measure $\rho$ on 
$\Ccal_\loc(I,H_\rw)$ 
such that $\Pi_t\rho =\mu_t$, for all $t\in I$? In other
words, when can we say that a statistical solution is a Vishik-Fursikov
statistical solution? We present below a criterion for 
when this is true or not in the case of an interval open on 
the left and such that $\mu_t$ is carried by a bounded set in $H$, uniformly in $t$.
  
First, let us prove the following localization result.
\begin{prop}
  \label{localizationofcarrierofprojections}
  A Borel probability measure $\rho$ on $\Ccal_\loc(I,H_\rw)$ is carried by the space
  $\Ccal_\loc(I,B_H(R)_\rw)$ if and only if $\rho_t$ is carried by $B_H(R)$ for $t$ in a 
  countable dense subset of $I$.
\end{prop}  
  
\begin{proof}
  Let $\rho$ be a Borel probability measure on $\Ccal_\loc(I,H_\rw)$.
  
  If $\rho$ is carried by $\Ccal_\loc(I,B_H(R)_\rw)$ and $t\in I$ is arbitrary, then, using that
  $\Pi_t^{-1}B_H(R)\supset\Ccal_\loc(I,B_H(R)_\rw)$, we find that
  \[ \rho_t(B_H(R)) = \rho(\Pi_t^{-1}B_H(R))\geq \rho(\Ccal_\loc(I,B_H(R)_\rw))=1, 
  \]
  so that $\rho_t$ is carried by $B_H(R)$ for all $t$ in $I$, and, hence, in particular,
  for $t$ in any countable dense subset of $I$.
  
  Now let us assume that $\rho_t$ is carried by $B_H(R)$, for $t$ in a countable
  dense subset $D$ of $I$.
  Consider $R'>R$ and let $\psi_{R'}(s)$ be the continuous
  positive function defined for $s\geq 0$ which is equal to $1$, for $0\leq s\leq R$, to $0$
  for $s\geq R'$, and is linear, for $R\leq s \leq R'$. 
  Let $\varphi(\bu)=\psi_{R'}(|P_m\bu|_{L^2})$
  for some integer $m$, where $P_m$ is the Galerkin projector defined in 
  Section \ref{NSEsettingsec}. Then, for any time $t\in D$,
  \begin{align*}
     \int_{\Ccal_\loc(I,H_\rw)}\psi_{R'}(|P_m\bv(t)|_{L^2})\;\rd\rho(\bv) 
      & = \int_{\Ccal_\loc(I,H_\rw)}\varphi(\bv(t))\;\rd\rho(\bv) \\
      & = \int_{H_\rw} \varphi(\bu)\;\rd\rho_{t}(\bu) \\
      & = \int_{H_\rw} \psi_{R'}(|P_m\bu|_{L^2})\;\rd\rho_{t}(\bu).
  \end{align*}
  Since $\rho_{t}$ is carried by $B_H(R)_\rw$ and $\psi_{R'}(|P_m\bu|_{L^2})=1$
  in this ball, we have
  \begin{align*}
    \int_{\Ccal_\loc(I,H_\rw)}\psi_{R'}(|P_m\bv(t)|_{L^2})\;\rd\rho(\bv) 
      & = \int_{B_H(R)_\rw} \psi_{R'}(|P_m\bu|_{L^2})\;\rd\rho_{t}(\bu) \\
      & = \int_{B_H(R)_\rw} \;\rd\rho_{t}(\bu) = 1.
  \end{align*}
  Letting $m\rightarrow\infty$, we obtain, thanks to the Lebesgue Dominated Convergence
  Theorem, that
  \begin{equation}
    \label{intpsiRprime}
    \int_{\Ccal_\loc(I,H_\rw)}\psi_{R'}(|\bv(t)|_{L^2})\;\rd\rho(\bv) = 1.
  \end{equation}
  Letting $R'$ decrease to $R$, the functions $\psi_{R'}$ decrease to the characteristic
  function of the interval $[0,R]$, so that $\bu\mapsto\psi_{R'}(|\bu|_{L^2})$ decreases to
  the characteristic function $\bu\mapsto\chi_{B_H(R)}(\bu)$. Hence, \eqref{intpsiRprime}
  implies, at the limit,
  \begin{equation}
    \label{intchit1}
    \int_{\Ccal_\loc(I,H_\rw)} \chi_{B_H(R)}(\bv(t))\;\rd\rho(\bv) = 1.
  \end{equation}
  Letting
  \[ \Ccal_I(t;R)=\left\{\bv\in\Ccal_\loc(I,H_\rw); \;|\bv(t)|_{L^2}\leq R\right\},
  \]
  we have that \eqref{intchit1} implies that
  \[ \rho\left(\Ccal_I(t;R)\right) = 1, \quad \forall t\in D.
  \]
  Consequently,
  \begin{equation}
    \rho\left(\bigcap_{t\in D} \Ccal_I(t;R)\right) =1.
  \end{equation}
  Notice that
\begin{equation}
  \bigcap_{t\in D} \Ccal_I(t;R)
    = \left\{\bv\in\Ccal_\loc(I,H_\rw); \;|\bv(t)|_{L^2}\leq R, \;\forall t\in D\right\}.
  \end{equation}
  Now, since $t\rightarrow |\bv(t)|_{L^2}$ is lower-semi-continuous, we have that
  \[ |\bv(t)|_{L^2}\leq R, \quad \forall t\in \RR, 
       \quad \forall \bv\in \bigcap_{t\in D} \Ccal_I(t;R).
  \]
  Thus, 
  \[ \bigcap_{t\in D} \Ccal_I(t;R) = \Ccal_\loc(I,B_H(R)_\rw),
  \]
  and, hence,
  \begin{equation}
    \rho(\Ccal_\loc(I,B_H(R)_\rw))=1,
  \end{equation}
  or, in other words, $\rho$ is carried by $\Ccal_\loc(I,B_H(R)_\rw)$.
\end{proof}  
  
Now, for any $R\geq R_0$, consider the set of statistical solutions (according to
Definition \ref{deftimedependetstatisticalsolution}) carried by the closed ball $B_H(R)$:
\begin{equation*}
  \Mcal_R = \left\{ \{\mu_t\}_{t\in I} \text{ is a statistical solution with } 
    \mu_t(B_H(R)) = 1, \;\forall t\in I \right\}.
\end{equation*}
One can show that, for such measures, the map 
\begin{equation}
  t\mapsto \int_H \varphi(\bu) \;\rd\mu_t(\bu) \in \Ccal([t_0,\infty)),
\end{equation}
is continuous on $I$, for any $\varphi\in\Ccal(B_H(R)_\rw)$.
Indeed, when $\varphi$ is of the form \eqref{cylindricalfunctions}, the continuous
dependence in $t$ of this integral follows from equation \eqref{liouvilleeq}. For 
a general $\varphi\in\Ccal(B_H(R)_\rw)$, we notice that $\varphi$ can be uniformly 
approximated by functions $\varphi_k$ of the type \eqref{cylindricalfunctions}, thanks
to the Stone-Weierstrass Theorem (see Section \ref{cylindricaltestfn}). 
The continuous dependence in $t$ of the
integrals $\int_H \varphi_k(\bu) \;\rd\mu_t(\bu)$ gives, at the limit, the
continuous dependence in $t$ of $\int_H \varphi_k(\bu) \;\rd\mu_t(\bu)$. 

In this set $\Mcal_R$, we define the sequential convergence
\[ \mu^{(n)} \rightarrow \mu, \quad \text{as } n\rightarrow \infty,
\]
by the condition that, for any $\varphi\in\Ccal(B_H(R)_\rw)$, 
\begin{equation}
  \label{convergenceinmr}
  \int_H \varphi(\bu) \;\rd\mu_t^{(n)}(\bu)
        \rightarrow  \int_H \varphi(\bu) \;\rd\mu_t(\bu), \quad \text{as } n\rightarrow \infty,
\end{equation}
uniformly in $t$ on any compact subset of $I$.

Now, we also define the set of Dirac delta-like statistical solutions, that is
statistical solutions supported on an individual (Leray-Hopf) weak solution of the
Navier-Stokes equations,
\[ \Dcal_R = \left\{ \{\delta_{\bu(t)}\}_{t\in I}; \; \bu \text{ is a
      weak solution on } I \text{ with } \bu(t)\in B_H(R), 
      \forall t\in I \right\},
\]
and the set of Vishik-Fursikov statistical solutions
\[ \mathcal{VF}_R = \left\{ \{\rho_t\}_{t\in I} \text{ is a Vishik-Fursikov statistical 
     sol. with } \rho_t(B_H(R)) = 1, \;\forall t\in I \right\}.
\]

It is straighforward to show that 
\begin{equation}
  \label{inclusionspacesofmeasures}
  \Dcal_R \subset \mathcal{VF}_R \subset \Mcal_R,
\end{equation}
and that $\mathcal{VF}_R$ and $\Mcal_R$ are convex. We then consider the
convex hull $\co\Dcal_R$ of $\Dcal_R$, which is the set of all finite convex combinations of
elements of $\Dcal_R$. Since $\mathcal{VF}_R$ and $\Mcal_R$ are convex, we have
\begin{equation}
  \label{inclusionspacesofmeasuresconvex}
  \co\Dcal_R \subset \mathcal{VF}_R \subset \Mcal_R.
\end{equation}

One can also show that $\mathcal{VF}_R$ and $\Mcal_R$ 
are closed for the sequential convergence defined above. Hence, considering the
closure of $\co\Dcal_R$ with respect to this sequential convergence (i.e. all the measures
obtained as the limits of sequences of finite convex combinations of measures in 
$\Dcal_R$), we have 
\begin{equation}
  \label{inclusionspacesofmeasuresconvexclosure}
  \overline{(\co\Dcal_R)} \subset \mathcal{VF}_R \subset \Mcal_R.
\end{equation}

The next result shows that $\mathcal{VF}_R$ is precisely the
closure of the convex hull of $\Dcal_R$.

\begin{thm}
  \label{charactvfr}
  Let $I\subset\RR$ be an interval open on the left and $R\geq R_0$. Then, 
  \[ \mathcal{VF}_R=\overline{(\co\Dcal_R)}.
  \] 
  In other words, a statistical solution in $\Mcal_R$
  is a Vishik-Fursikov statistical solution
  if and only if it is the limit, in the sense of \eqref{convergenceinmr}, of a sequence of 
  finite convex combinations of measures in $\Dcal_R$. 
\end{thm}

\begin{proof}
  It is not difficult to check that $\Dcal_R$ is closed in $\Mcal_R$ with respect to
  the convergence of sequences in $\Mcal_R$ according to \eqref{convergenceinmr}.
  Then, due to \eqref{inclusionspacesofmeasures}, we need to show that
  \begin{equation}
    \label{inclusionneedobeshowfomeasures}
    \mathcal{VF}_R \subset \overline{(\co\Dcal_R)}.
  \end{equation}  
  In other words, we need to show that if a statistical solution belongs to $\mathcal{VF}_R$,
  then it can be
  approximated, in the sense above, by a sequence of convex combinations
  of Dirac measures. Let then $\{\rho_t\}_{t\in I}$ be a Vishik-Fursikov statistical
  solution such that $\rho_t=\Pi_t\rho$, for all $t\in I$, for some $\rho\in\Ccal_\loc(I,H_\rw)$
  with $\rho_t(B_H(R))=1$, for all $t\in I$.
  
  Thanks to Proposition \ref{localizationofcarrierofprojections}, we know that $\rho$ is 
  carried by $\Ccal_\loc(I,B_H(R)_\rw)$, and, hence, it is carried by $\Ucal_I^\sharp(R)$.

  Since $\Ucal_I^\sharp(R)$ is a compact Polish space, and thus a
  compact separable space, and $\rho$ is a Borel probability measure on 
  $\Ucal_I^\sharp(R)$, it follows from the Krein-Milman Theorem that each $\rho$ 
  is the limit of a sequence of convex combinations of Dirac deltas in $\Ucal_I^\sharp(R)$, 
  i.e.
  \begin{equation}
    \label{pcalucalikreinmilman}
    \sum_j \theta_j^{(n)} \delta_{\bv_j^{(n)}} \stackrel{*}{\rightharpoonup} \rho,
  \end{equation} 
  with $\theta_j^{(n)} \geq 0$, $\sum_j \theta_j^{(n)} = 1$, $\bv_j^{(n)}\in \Ucal_I^\sharp(R)$.
  
  The convergence \eqref{pcalucalikreinmilman} means that, for every 
  $\Phi\in\Ccal(\Ucal_I^\sharp(R))$, we have
  \[ \sum_j \theta_j^{(n)}\Phi(\bv_j^{(n)}) \rightarrow \int_{\Ucal_I^\sharp(R)} \Phi(\bv)
      \;\rd\rho(\bv).
  \]
  Taking $\Phi(\bv)=\varphi(\bv(t))$, for $\varphi\in \Ccal(B_H(R))$ and $t\in I$,
  define
  \[  \psi_n(t) = \sum_j \theta_j^{(n)} \varphi(\bv_j^{(n)})
  \]
  and 
  \[ \psi(t) = \int_{H} \varphi(\bu)\;\rd\rho_t(\bu),
  \]
  and notice that
  \begin{multline*}
    \psi_n(t) = \sum_j \theta_j^{(n)} \varphi(\bv_j^{(n)}) \rightarrow
      \int_{\Ucal_I^\sharp(R)} \varphi(\bv(t))\;\rd\rho(\bv) \\
      = \int_{H} \varphi(\bu)\;\rd\rho_t(\bu) = \psi(t), \quad \forall t\in I.
  \end{multline*}
  
  Since the weak solutions in $\Ucal_I^\sharp(R)$ are weakly 
  equicontinuous on any compact interval $J\subset I$ with values in $H_\rw$,
  we have that the $\psi_m$ are equicontinuous, hence their pointwise
  convergence to $\psi$ on $J$ implies their uniform convergence to $\psi$
  on $J$, for every compact interval $J\subset I$. This means that 
  $\{\sum_j \theta_j^{(n)} \delta_{\bv_j^{(n)}(t)}\}_t$ converges to $\{\rho_t\}_t$
  in the sense of \eqref{convergenceinmr}, which concludes the proof
  of \eqref{inclusionneedobeshowfomeasures} and of Theorem \ref{charactvfr}.
\end{proof}

We now have the following characterization of the Vishik-Fursikov statistical 
solutions.

\begin{cor}
  Let $I\subset \RR$ be an interval open on the left.
  Let $\mu=\{\mu_t\}_{t\in I}$
  be a statistical solution in $\Mcal_R$, for some $R\geq R_0$. Then, $\{\mu_t\}_{t\in I}$
  is a Vishik-Fursikov statistical solution if and only if there exists a sequence
  $\{\rho^{(n)}\}_n$ of convex combinations of statistical solutions of the form
  $\{\delta_{\bv(t)}\}_{t\in I}$, where $\bv\in \Ucal_I^\sharp(R)$, such that
  $\rho^{(n)}\rightarrow \mu$ in $\Mcal_R$.
\end{cor}

\begin{proof}
  Since any family of the form $\{\delta_{\bv(t)}\}_{t\in I}$ is a Vishik-Fursikov
  statistical solution, any convex combination of them, as in the statement of
  the corollary, is a Vishik-Fursikov statistical solution as well, and, therefore,
  so is its limit. This last fact as well as the converse property follow directly from 
  Theorem \ref{charactvfr}.
\end{proof}

\begin{rmk}
  In the case of measures with unbounded support, not belonging to $\Mcal_R$
  for any $R>0$, we consider, for $E>0$, the set
  \begin{equation}
    \Mcal_{(E)} = \left\{ \{\mu_t\}_{t\in I} \text{ is a statistical solution };
       \;\int_H |\bu|_{L^2}^2\;\rd\mu_t(\bu) \leq E, \;\forall t\in I \right\}.
  \end{equation}
  We also define in a similar manner the sets $\Dcal_{(E)}$ and $\mathcal{VF}_{(E)}$.
  Then,
  \[ \Dcal_{(E)} \subset \co\Dcal_{(E)} \subset \mathcal{VF}_{(E)} 
       \subset \Mcal_{(E)}, 
  \]
  and the analog of Theorem \ref{charactvfr} would be that every statistical solution in 
  $\Mcal_{(E)}$ is a Vishik-Fursikov statistical solution if and only if it is the limit,
  in some suitably defined sense, of measures in $\co\Dcal_{(E)}$. However, this analog
  result remains a conjecture at this point and depends, in part, on an appropriate
  definition of convergence for such measures.
  \qed
\end{rmk}

\end{document}